\documentclass[11pt]{amsproc}
\usepackage[mathcal]{euscript}
\usepackage{amsthm}
\usepackage{amsfonts}
\usepackage{amsmath,amssymb}
\usepackage{indentfirst,latexsym,bm,amsthm,graphicx,colortbl}
\usepackage{fancyhdr}
\usepackage{times}
\usepackage{amsmath,amsfonts,amssymb,amsthm}
\usepackage{epsfig}
\usepackage[all]{xy}
\newtheorem{theorem}{Theorem}[section]
\newtheorem{lemm}[theorem]{Lemma}
\newtheorem{prop}[theorem]{Proposition}
\newtheorem{theo}[theorem]{Theorem}
\theoremstyle{definition}
\newtheorem{defi}[theorem]{Definition}

\newtheorem{coro}[theorem]{Corollary}
\theoremstyle{remark}
\newtheorem{remark}[theorem]{Remark}

\numberwithin{equation}{section}

\newcommand{\diag}{\mbox{diag}}

\usepackage{tikz}
\usetikzlibrary{graphs,matrix}
\usepgflibrary{arrows}

\begin{document}

\title[Quantum $N$-toroidal algebras and quantized GIM algebras]
{Quantum $N$-toroidal algebras and extended quantized GIM algebras of $N$-fold affinization\\
\vspace{12pt}{\itshape\small Dedicated to R.V. Moody on the occasion of his 80th birthday}}

\author[Gao]{Yun Gao}
\address{Department of Mathematics and Statistics,
   York University, Toronto, ON M3J 1P3,
   Canada}
\email{ygao@yorku.ca}

\author[Jing]{Naihuan Jing}
\address{Department of Mathematics,
   North Carolina State University,
   Ra\-leigh, NC 27695, USA}
\email{jing@ncsu.edu}

\author[Xia]{Limeng Xia}
\address{Institute of Applied System Analysis, Jiangsu University,
Zhenjiang, Jiangsu 212013, China} \email{xialimeng@ujs.edu.cn}

\author[Zhang]{Honglian Zhang}
\address{Department of Mathematics, Shanghai University,
Shanghai 200444, China} \email{hlzhangmath@shu.edu.cn}



\begin{abstract}
We introduce the notion of
quantum $N$-toroidal algebras 
as natural generalization of the quantum toroidal algebras as well as extended quantized GIM algebras of $N$-fold affinization. We show that the quantum $N$-toroidal algebras are quotients of the extended quantized GIM algebras of $N$-fold affinization, which generalizes a well-known result of Berman and Moody for Lie algebras.

\end{abstract}
\maketitle

\section{ Introduction}

One of the most important generalizations of the finite dimensional complex Lie algebra ${\mathfrak g}$ is the (untwisted) affine Lie algebra $\widehat{\mathfrak g}$ (without derivation), the central extension of ${\mathfrak g}\otimes \mathbb C[t_0^{\pm1}]$ by the one-dimensional center $\mathbb{C}c_0$.
 The $N$-toroidal Lie algebra ${\mathfrak g}_{N,tor}$ is a further
 generalization and the infinite dimensional universal central extension of $\mathfrak g\otimes\mathbb C[t_0^{\pm1},\cdots, t_{N-1}^{\pm1}]$ (c.f. \cite{RM} etc.). The algebra
 ${\mathfrak g}_{2,tor}={\mathfrak g}_{tor}$ is usually called the toroidal Lie algebra or simply referred to as the double affine Lie algebra.
   The $N$-toroidal Lie algebra ${\mathfrak g}_{N,tor}$ has close connections with algebraic geometry, finite groups, conformal field theory, vertex algebras, Yangians, and differential equations and so on, and there are extensive works on the general toroidal Lie algebra (c.f. \cite{ABFP}).

The quantum toroidal algebra $U_q({\mathfrak g}_{tor})$=$U_q({\mathfrak g}_{2,tor}$) in type A was introduced by  Ginzburg, Kapranov, and Vasserot
\cite{GKV} in connection with geometric realization
and Langlands reciprocity for algebraic surfaces. Besides the realization of Hecke operators for vector bundles on algebraic surfaces,
Varagnolo and Vasserot \cite{VV1} proved a Schur-Weyl duality between
representations of the quantum toroidal algebras  $U_q({\mathfrak g}_{tor})$ and elliptic Cherednik algebras. Vertex representations of the
quantum toroidal algebras in ADE types were also realized via the McKay correspondence \cite{FJW}.
In a series of papers \cite{M1}-\cite{M6}, Miki studied the structures and representations of the quantum toroidal algebra $U_q({\mathfrak g}_{tor})$ exclusively in type A.
 In \cite{GJ}, the authors constructed explicitly an irreducible vertex representation of the quantum toroidal algebra $U_q({\mathfrak g}_{tor})$ of type $A$ on the basic module for the affine Lie algebra $\hat{\mathfrak{gl}}_N$.
In the review \cite{H2}, the representation theory of general
quantum toroidal algebras $U_q({\mathfrak g}_{tor})$ was understood as quantum affinizations (see also \cite{J2}).
Numerous important works on the quantum toroidal algebras and quantum affinizations were carried out in \cite{STU}, \cite{Sy}, \cite{H1}, \cite{H2}, \cite{VV2}, \cite{Na}, \cite{Na2}, \cite{FJM1}, \cite{FJM2}, \cite{GTL}, \cite{GM}, \cite{GNW} etc. 
Note that most of these works studied the structures and representations of the quantum toroidal algebra in type $A$, which
further admits a two-parameter deformation $U_{q, \kappa}(\mathfrak{g}_{tor})$.
Despite of all these, it is still far 
from complete understanding of the quantum toroidal algebras in type A, and even less is known for
the representation theory of the quantum toroidal algebras in other types.

In \cite{BM}, the authors introduced the generalized intersection matrix (GIM) of N-fold affinization. The 
 GIM algebra is defined by Chevalley generators subject to
Serre-like relations defined by 
the GIM matrix, which is similar to the Cartan matrix but with (possible) positive off-diagonal entries 
(c.f. \cite{Sl,Sk}). The N-toroidal Lie algebras were then proved to be quotient algebras of the GIM algebra of N-fold affinization (c.f. Proposition 4.15 in \cite{BM}).
The quantum GIM algebra was introduced \cite{K} based on its relation with the
2-fold affinization found in \cite{T1, T2, LT}, however, it is still mysterious in general on its relation with a possible quantum N-toroidal algebra.
Furthermore, we notice that the quantized GIM algebras for simply-laced cases are isomorphic to subalgebras of  the quantum universal enveloping algebras \cite{GHX} .

This leads to an important question on how to
generalize the quantum toroidal algebra $U_q({\mathfrak g}_{tor})$ to the  quantum N-toroidal algebra $U_q({\mathfrak g}_{N,tor})$ for
general $N$
and uncover their relations with other important algebraic structures such as quantum GIM algebras. 
In the present paper, we introduce the notion of
quantum $N$-toroidal algebras for all types uniformly as natural generalization of the usual quantum toroidal algebra,
just like the relation between $2$-toroidal Lie algebras and $N$-toroidal Lie algebras.
 We find out that the novel quantum $N$-toroidal algebras are closely related to general extended quantized GIM algebras of $N$-fold affinization by using some simplified Drinfeld-type generators.

 In \cite{JZ1} and \cite{JZ2}, we formulated a simplified set of Drinfeld generators for the quantum affine algebras and quantum toroidal algebras in type $A$, respectively, to simplify
 practical computations.
In the same way, the quantum $N$-toroidal algebra can be realized as a quotient algebra of certain quantum algebra generated by simplified generators.
Interestingly, this formulation leads to an identification of  
the quantum $N$-toroidal algebra as a quotient algebra of the extended quantized GIM algebra of $N$-affinization, which is 
consistent with the case of Lie algebras \cite{BM}. The corresponding GIMs of N-fold affinization, as well as the  Dynkin diagrams for $N=3$ of the subalgebras were given case by case.
Furthermore we can realize our new algebras as
certain subalgebras of the quantum toroidal algebras, thus showing that
our new algebras have nontrivial (vertex) representations. 
We remark that in type $A$ the vertex representation can also be constructed using the
two-parameter deformation $U_{q,\kappa}(\mathfrak{g}_{N, tor})$, and the latter algebra is a generalization of the quantum $2$-toroidal algebra for type $A$ given in \cite{Sy}.

On the other hand, we note that the quantum GIM algebras carry nontrivial finite dimensional representations \cite{X}, while it is known that the quantum toroidal algebras don't have
such representations unless the centers are trivial. This shows that the newly defined quantum $N$-toroidal algebras may help with further investigation on their categorification.

In section
3, we define the quantum $N$-toroidal algebra for all types uniformly. At the same time, we find a subset of Drinfeld generators for the quantum $N$-toroidal algebra. It is shown 
that the algebra generated by this subset can be realized as the quotient of the extended quantized GIM algebra of $N$-fold affinization.
It turns out that the quantum $2$-toroidal algebra 
is isomorphic to a quotient algebra of the algebra for type $A$ and equals to the algebra for other types.  
In general the quantum $N$-toroidal algebra for $N\geq 2$	 is isomorphic to the quotient algebra of the algebra.  This main result will be verified in the next two sections for $N=2$ and $N\geq3$, respectively, which implies that the quantum $N$-toroidal algebra
is isomorphic to the quotient algebra of the extended quantized GIM algebra of $N$-affinization.
In section 6, a vertex realization of the quantum $N$-toroidal algebra is given.
 In the Appendix, we list the Dynkin diagrams of the GIMs of $N$-fold affinization case by case for $N=2$ and $N=3$ for example.

\section{Extended quantized GIM algebras of $N$-fold affinization}
In this section, we first
recall the definition of a generalized intersection matrix (GIM for short)\, (c.f. \cite{Ne},\cite{G})
and then give a general definition of extended quantized GIM algebras of $N$-fold affinization
(cf. \cite{K}).


\begin{defi}  Let $J$ be a finite index set, a square matrix $M=(m_{ij})_{i, \, j\in
J}$ over $\mathbb{Z}$ is called a generalized intersection matrix if  it  satisfies:

(C1) \qquad\quad $m_{ii}=2$  for $~i\in J;$

(C2) \qquad\quad $m_{ij}\cdot m_{ji} $ are nonnegative  integers for ~$i\neq j$;

(C3) \qquad\quad $m_{ij}=0$  implies ~$m_{ji}=0$.
\end{defi}\begin{remark} As $m_{ij}$ can be positive, the notion of GIM generalizes that of a generalized Cartan matrix.
\end{remark}

In this paper, we only consider the symmetrizable intersection matrix (IM) $M=(m_{ij})_{i, \, j\in
J}$, i.e. there exists an integral diagonal matrix $D$ such that $DM$ is symmetric. We fix the notation $D=\diag(d_i\in\mathbb{Z}_+|i\in J)$.

The GIM algebras were introduced by P. Slodowy as  generalization of the Kac-Moody Lie algebras \cite{Sl}.
Similar to the latter, 
a GIM algebra $\mathcal{L}(M)$ associated to a GIM  $M=(m_{ij})$  
can be defined by generators and relations (c.f. \cite{BM}).

 \begin{defi}  The GIM algebra $\mathcal{L}(M)$ associated to a GIM  $M=(m_{ij})_{i, \, j\in
J}$  is the Lie algebra  over $\mathbb{C}$ generated by $e_i, f_i, h_i$ for $i\in J$ satisfying the following relations,

 \noindent $(R1)$ \ For $\,i, j\in J$,
\begin{equation*}
\begin{array}{ll}
&[h_i, e_j]=m_{ij}e_j,\qquad\quad
[h_i, f_j]=-m_{ij}f_j,\qquad \quad [e_i, f_i]=h_i.
\end{array}
\end{equation*}

\noindent $(R2)$ \ For $m_{ij}\leqslant 0$,
\begin{equation*}
\begin{array}{ll}
&[e_i, f_j]=0=[f_i, e_j], \qquad
(ade_i)^{-m_{ij}+1}e_j=0=(adf_i)^{-m_{ij}+1}f_j.
\end{array}
\end{equation*}

\noindent $(R3)$ \ For $m_{ij}>0$ and $i\neq j$,
\begin{equation*}
\begin{array}{ll}
&[e_i, e_j]=0=[f_i, f_j],\qquad
(ade_i)^{m_{ij}+1}f_j=0=(adf_i)^{m_{ij}+1}e_j.
\end{array}
\end{equation*}
 \end{defi}

Let $I_0=\{1, 2, \cdots, n\}$ and $\tilde{J}=\{-N+1,\cdots, -1, 0, 1, \cdots, n\}$.

\begin{defi}\label{defi3.10} \,Let $A=(a_{ij})_{i, j\in I_0}$ be a Cartan matrix of finite type. Define
$$M=(m_{ij})_{i,j\in \tilde{J}}=\left(\begin{array}{lr}
T & P\\
Q & A
\end{array}\right),$$
where $T$ is the $N\times N$ matrix $\sum\limits_{i,j}2E_{ij}$, and $P=(p_{ij})$ (resp. $Q=(q_{ij})$) is the $N\times n$ (resp. $n\times N$) matrix given by $p_{ij}=a_{0j}$ (resp. $q_{ij}=a_{i0}$).
\end{defi}


\begin{remark}\,
Note that $M$ is an $N$-fold affinization of $A$, and is exactly the GIM introduced in \cite{BM} after reordering the index.
\end{remark}

In \cite{T1, T2, LT, GHX}, the authors studied the quantized GIM algebras for simply-laced cases associated to a GIM of 2-fold affinization. We will study
a more general algebraic structure, namely the extended quantized GIM algebra $U_{q}(\mathcal{L}(M))$  associated to a GIM of $N$-fold affinization $M$ for any finite simple type, which will be called the extended quantized GIM algebra of $N$-fold affinization for simplicity. Here  ``extended'' refers to adding a derivation to the algebra. 
Let $\tilde{J}=J_1\cup J_2$ be a disjoint decomposition with $card(J_1)=N$ and $\mathbb{K}=\mathbb{C}(q)$.

\begin{defi}  The extended quantized  GIM algebra
$U_{q}(\mathcal{L}(M))$  of  $N$-fold affinization is the unital associative
algebra over $\mathbb{K}$ generated by the elements $E_i,\, F_i,\,
K_i^{\pm 1}, q^{\pm d}  (i\in \tilde{J})$,
satisfying the following relations:

\noindent $(M1)$ \ For $\,i, j\in \tilde{J}$,
$K_i^{\pm1}\,K_i^{\mp1}=1$,  $q^{\pm d}$ and $K_j^{\pm1}$ commute with each other.\\
\noindent$(M2)$ \ For $\,i\in J_1$ and $j\in J_2$,
\begin{equation*}
\begin{array}{ll}
&q^d\,E_i\,q^{\,-d}=q\,E_i,\qquad\quad
q^d\,F_i\,q^{\,-d}=q^{-1}\,F_i,\\
&q^d\,E_j\,q^{\,-d}=\,E_j,\qquad\quad
q^d\,F_j\,q^{\,-d}=\,F_j.\\
\end{array}
\end{equation*}\\

\noindent$(M3)$ \ For $\,i\in \tilde{J}$ and $j\in \tilde{J}$,
\begin{equation*}
\begin{array}{ll}
&K_i\,E_j\,K_i^{\,-1}=q_i^{m_{ij}}\,E_j,\qquad\quad
K_i\,F_j\,K_i^{\,-1}=q_i^{-m_{ij}}\,F_j.
\end{array}
\end{equation*}\\
$(M4)$ \ For $\,i\in \tilde{J}$, we have that
 $$[\,E_i, F_i\,]=\frac{K_i-K_i^{-1}}{q_i-q_i^{-1}}.$$
$(M5)$ \ For $m_{ij}<0$, we have that
\begin{equation*}
\begin{array}{ll}
&[\,E_i, F_j\,]=0,\vspace{6pt}\\
&\sum\limits_{s=0}^{1-m_{ij}}(-1)^{s}\Big[{1-m_{ij}\atop
s}\Big]_i
E_{i}^{1-m_{ij}-s}E_jE_{i}^{s}=0, \vspace{6pt}\\
&\sum\limits_{s=0}^{1-m_{ij}}(-1)^{s}\Big[{1-m_{ij}\atop
s}\Big]_i
F_{i}^{1-m_{ij}-s}F_jF_{i}^{s}=0.
\end{array}
\end{equation*}
$(M6)$ \ For $m_{ij}>0$ and $i\neq j$, we have that
\begin{equation*}
\begin{array}{ll}
&[\,E_i, E_j\,]=0=[F_i, F_j],\vspace{6pt}\\
&\sum\limits_{s=0}^{1+m_{ij}}(-1)^{s}\Big[{1+m_{ij}\atop
s}\Big]_i
E_{i}^{1+m_{ij}-s}F_jE_{i}^{s}=0, \vspace{6pt}\\
&\sum\limits_{s=0}^{1+m_{ij}}(-1)^{s}\Big[{1+m_{ij}\atop
s}\Big]_i
F_{i}^{1+m_{ij}-s}E_jF_{i}^{s}=0.
\end{array}
\end{equation*}
$(M7)$\ For $m_{ij}=0$ and $i \neq j$, we have that
\begin{gather*}
[\,E_i, E_j\,]=0=[E_i, F_j]=[F_i, F_j],
\end{gather*}
where $q_i=q^{d_i}$, $[m]_{i}=\frac{q_i^{m}-q_i^{-m}}{q_i-q_i^{-1}}$, $[m]_{i}!=[m]_{i}\cdots[2]_{i}[1]_{i}$,
$\Bigl[{m\atop n}\Bigr]_{i}=\frac{[m]_{i}!}{[n]_{i}![m-n]_{i}!}$.
\end{defi}

\section{Quantum $N$-toroidal algebras $U_{q}$(${\mathfrak g}_{N,tor}$) $(N\geq 2)$}

\subsection{Quantum $q$-bracket}\, We recall the quantum $q$-brackets for later use (c.f. \cite{J1}).  For $v_i\in \mathbb{K}\backslash \{0\} (i=1,\cdots, s-1)$,
the quantum $q$-bracket $[\,a_1, a_2,\cdots,
a_s\,]_{(v_1, \cdots,\, v_{s-1})}$ is defined
inductively by
\begin{gather*}    [\,a_1, a_2\,]_{v_1}=a_1a_2-v_1\,a_2a_1,\\
[\,a_1, a_2, \cdots, a_s\,]_{(v_1,\,v_2,\,\cdots,
	\,v_{s-1})}=[\,a_1, \cdots, [a_{s-1},
a_s\,]_{v_1}]_{(v_2,\,\cdots,\,v_{s-1})}.
\end{gather*}
It follows immediately that
\begin{eqnarray}
&& [\,a,[\,b,c\,]_u\,]_v=[\,[\,a,b\,]_q,
c\,]_{\frac{uv}q}+q\,[\,b,[\,a,c\,] _{\frac{v}q}\,]_{\frac{u}q},\label{v1}\\
&&[\,[\,a,b\,]_u,c\,]_v=[\,a,[\,b,c\,]_q\,]_{\frac{uv}q}+q\,[\,[\,a,c\,]
_{\frac{v}q},b\,]_{\frac{u}q}.\label{v2}
\end{eqnarray}

\subsection{Quantum $N$-toroidal algebra $ U_{q}(\mathfrak{g}_{N, tor})$ via generating functions}
\, The quantum toroidal algebra $U_{q, \kappa}(\mathfrak{sl}_{n+1},{tor})$ for type A  was introduced  in \cite{GKV} as a two-parameter deformation. It admits the Schur-Weyl duality \cite{VV1}.  In \cite{H2},  the quantum toroidal algebras $U_{q}(\mathfrak{g}_{tor})$ for general types were introduced as Drinfeld quantum affinizations.
We will define the quantum $N$-toroidal algebra uniformly through the process of Drinfeld-like quantum $N$-affinization. In particular, the new algebra
is a natural generalization of
the quantum toroidal algebras $U_{q}(\mathfrak{g}_{tor})$ (see \cite{H2} etc.).

Let $I=\{0, 1, \cdots, n\}$ and $I_0=\{1, \cdots, n\}$. Set $\mathfrak{g}$ a complex simple Lie algebra of rank $n$,  $\alpha_i ~(i\in I_0)$ the simple roots of  $\mathfrak{g}$ and $\hat {\mathfrak{g}}$
the non-twisted affine Lie algebra associated to $\mathfrak{g}$. Let  $A=(a_{ij})_{i,j\in I_0}$ be the Cartan matrix of $\mathfrak{g}$ and $\mathfrak{h}$ the Cartan subalgebra.
Let $\delta$ denote the primitive imaginary root  of the affine Lie algebra $\hat {\mathfrak{g}}$ and $\theta$ the highest root of $\mathfrak{g}$, take $\alpha_0=\delta-\theta$,  then $\Pi=\{\alpha_i\mid i\in I\}$ is a basis of simple roots of the affine Lie algebra $\hat{\mathfrak {g}}$.

 Let  $\hat{A}=(a_{ij})_{i, j\in I} $ be the generalized Cartan matrix
of the affine Lie algebra $\hat{\mathfrak{g}}$, and $\hat{\mathfrak{h}}$ the Cartan subalgebra of $\hat{\mathfrak{g}}$. There exists a diagonal matrix $D=\diag(d_i|i\in I)$  
such that $D\hat{A}$ is symmetric. The non-degenerate symmetric bilinear form $(\hspace{0.4cm}|\hspace{0.4cm})$ on $\hat{\mathfrak{h}}^*$ satisfies  for all $i,j\in I$,
\begin{eqnarray*}
(\alpha_i|\alpha_j)=d_ia_{ij},\quad (\delta|\alpha_i)=(\delta|\delta)=0,
\end{eqnarray*}
where $\hat{\mathfrak{h}}^*$ denotes the dual Cartan subalgebra of $\hat{\mathfrak{g}}$.

Let $q_i=q^{d_i}$,
 $J=\{1, \cdots {N-1}\}$, $\underline{k}=(k_1, k_2, \cdots, k_{N-1})\in \mathbb{Z}^{N-1}$, and $e_s=(0, \cdots, 0, 1, 0,$ $\cdots, 0)$ the $s$th  standard unit   vector of $(N-1)$-dimension. We also denote by $\underline{{0}}$ the $(N-1)$-dimensional zero vector.

\begin{defi}\label{defi-ntor}
	The  quantum $N$-toroidal algebra
	${U}_{q}$(${\mathfrak g}_{N,tor}$) is the associative algebra over $\mathbb{K}$
	generated by $x_{i}^{\pm}(\underline{k})$,  $a_i^{(s)}(\ell)$, $K_i^{\pm1}$,
	$\gamma_{s}^{\pm\frac{1}{2}}$,
	$q^{\pm d}$,  $( i\in I, s\in J$,
	${k}\in \mathbb{Z}$, $\ell \in \mathbb{Z}\backslash
	\{0\})$ satisfying the following relations,
	\begin{align}\label{n:Dr1}
	&\gamma_{s}^{\pm\frac{1}{2}}~~
	\textrm{are central such that} \gamma_s^{\frac{1}{2}}\gamma_s^{-\frac{1}{2}}=1~~ \textrm{and}~~ K_i^{\pm 1}\,K_i^{\mp 1}=1=q^{\pm d}q^{\mp d}, 
	\\ \label{n:Dr2}
	&[\,a_i^{(s)}(\ell),~K_{j}^{{\pm }1}\,]=0=[K_j^{\pm1}, q^{\pm d}]=[a_i^{(s)}(\ell), q^{\pm d}],
	\\ \label{n:Dr3}
	&[\,a_i^{(s)}(\ell),a_{j}^{(s')}(\ell')\,]
	=\delta_{\ell+\ell',0}\frac{
		[\,\ell
		a_{i{j}}\,]_i}{\ell}
	\cdot\frac{(\gamma_s\gamma_{s'})^{\frac{\ell}{2}}-(\gamma_s\gamma_{s'})^{-\frac{\ell}{2}}}{q_j-q_j^{-1}}
	,\\  
	\label{n:Dr4}
	& q^{d}\,x_{i}^{\pm}(\underline{k})\,q^{-d}=q^{\pm \delta_{i,0}}\, x_{i}^{\pm}(\underline{k}),\qquad
	\\\label{n:Dr4-1}
	&K_i\,x_{j}^{\pm}(\underline{k})\, K_i^{-1} = q_i^{\pm a_{i{j}}}x_{j}^{\pm}(\underline{k}), 
	\\ \label{n:Dr5}
	&[\,a_i^{(s)}(\ell),x_{j}^{\pm}(\underline{k})\,]=\pm \frac{ [\ell a_{i{j}}]_i}{\ell} \gamma_{s}^{\mp\frac{|\ell|}2}x_{j}^{\pm}(\underline{k}+\ell e_{s}), \qquad
	\\ \label{n:Dr6}
&[x_i^{\pm}(ke_s),x_{i}^{\pm}(ke_{s'})\,]=0, \\\label{n:Dr7}
&[x_i^{\pm}((t+1)e_s),x_{j}^{\pm}(t'e_{s})\,]_{q_i^{\pm a_{i{j}}}}+[x_{j}^{\pm}(({t'}+1)e_{s}),x_i^{\pm}(te_s)\,]_{q_i^{\pm a_{i{j}}}}=0,\\\label{n:Dr8}
&[\,x_i^{+}(te_s),~x_{j}^-(t'e_s)\,]=\frac{\delta_{i{j}}}{q_i-q_i^{-1}}\Big({\gamma_s}^{\frac{t-t'}{2}}\,
	\phi_i^{(s)}(t{+}t')-\gamma_s^{\frac{t'-t}{2}}\,\varphi_i^{(s)}(t{+}t')\Big),
	\end{align}
	where $\phi_i^{(s)}(t)$, $\varphi_i^{(s)}(-t)~(t\in \mathbb{Z}_{\geq 0})$ such that
	{$\phi_i^{(s)}(0)=K_i$,  $\varphi_i^{(s)}(0)=K_i^{-1}$} are defined as below:
	\begin{gather*}\label{n:Dr9}\sum\limits_{m=0}^{\infty}\phi_i^{(s)}(m) z^{-m}=K_i \exp \Big(
	(q_i{-}q_i^{-1})\sum\limits_{\ell=1}^{\infty}
	a_i^{(s)}(\ell)z^{-\ell}\Big),\\ \label{n:Dr10}
	\sum\limits_{m=0}^{\infty}\varphi_i^{(s)}(-m) z^{m}=K_i^{-1} \exp
	\Big({-}(q_i{-}q_i^{-1})
	\sum\limits_{\ell=1}^{\infty}a_i^{(s)}(-\ell)z^{\ell}\Big),
	\end{gather*}
	\begin{align}\label{n:Dr11}
		& \rm{Sym}_{{m_1},\cdots
		{m_{n}}}\sum_{k=0}^{n=1-a_{i{j}}}(-1)^k
	\Big[{n\atop  k}\Big]_{i}x_i^{\pm}({m_1e_s})\cdots x_i^{\pm}({m_ke_s}) x_{j}^{\pm}({\ell e_s})\\ \nonumber
	&\hskip1.8cm \times x_i^{\pm}({m_{k+1}e_s})\cdots x_i^{\pm}({m_{n}e_s})=0,
	\quad   i\neq j~~~~ \hbox{and}~~~~ s\in J,
\\ \label{n:Dr12}
	&\sum_{k=0}^{3}(-1)^k
	\Big[{3\atop  k}\Big]_{i}x_i^{\pm}({m_1 e_s})\cdots x_i^{\pm}({m_{3-k}e_s}) x_{i}^{\mp}({\ell e_{s'}})\\ \nonumber
	&\hskip1.2cm \times x_i^{\pm}({m_{4-k}e_s})\cdots x_i^{\pm}({m_{k}e_s})=0,
	~~\hbox{for}~~  i\in I~~ \hbox{and}~~m_1m_2m_3\ell\neq 0,\, s\neq s'\in J,
	\end{align}
	where  
	$\rm{Sym}_{{m_1},\cdots,
		{m_{n}}}$  denotes the symmetrization with respect to the indices $({m_1},\cdots,
		{m_{n}})$.
\end{defi}

\begin{remark}\,
	When $N=2$, Definition 3.1 is just that of the quantum toroidal algebra (\cite{H2} etc.).
Note that for the case of $\hat{\mathfrak{sl}}_{n+1}$,  an additional parameter can be added in the definition of the quantum toroidal algebra \cite{GKV}.
\end{remark}
\begin{remark}\,
	For each fixed $s\in J$, let 
$U_q^{(s)}$  be the subalgebra of ${U}_{q}({\mathfrak g}_{N,tor})$ generated by the elements  $x_i^{\pm}(ke_s)$, $a_i^{(s)}(\ell)$, $K_i^{\pm1}$,
	$\gamma_s^{\pm\frac{1}{2}}$,  $q^{\pm d}$ for $i\in I$, then 
$U_q^{(s)}$ is isomorphic to the quantum toroidal algebra $U_{q}({\mathfrak g}_{2,tor})$.
\end{remark}

\begin{remark}\,\label{remark3.4} There exists another central element $\gamma_0=K_{0}K_{\theta}$, where $\theta$ is the highest root of the simple Lie algebra $\mathfrak{g}$.
\end{remark}

\subsection{Simplified generators and the algebra $\mathcal{U}_0({\mathfrak g}_{N,tor})$ }\,
 In this subsection, we define an algebra  $\mathcal{U}_0({\mathfrak g}_{N,tor})$ generated by finite Drinfeld generators with finitely many Drinfeld relations and we prove that the quantum $N$-toroidal algebra
	$U_{q}({\mathfrak g}_{N,tor})$  is isomorphic to a quotient  of $\mathcal{U}_0$(${\mathfrak g}_{N,tor})$ or $\mathcal{U}_0$(${\mathfrak g}_{N,tor})$ itself (c.f. Theorem \ref{maintheo2-1} and  Theorem \ref{maintheo2-2}).
We will prove these theorems in the next two subsections respectively.

It is easy to see that the elements $x_i^{\pm}(\underline{0})$, $x_{0}^{-\epsilon}(\epsilon e_s)$,  $K_i^{\pm1}$, $q^{\pm d}$    and  $\gamma_{s}^{\pm\frac{1}2}$  ($\epsilon=\pm1$ or $\pm$, $i\in I$, $s\in J$) form a reduced set of generators for the algebra
${U}_q({\mathfrak g}_{N,tor})$.
\begin{defi}\label{def3.5} Denote by $\mathcal{U}_0({\mathfrak g}_{N,tor})$ the associative algebra over $\mathbb{K}$ generated by
	$x_i^{\pm}(\underline{0})$, $x_{0}^{-\epsilon}(\epsilon e_s)$,  $K_i^{\pm1}$, $q^{\pm d}$    and  $\gamma_{s}^{\pm\frac{1}2}$  ($\epsilon=\pm1$ or $\pm$, $i\in I$, $s\in J$)  satisfying the following relations:
	\begin{align}\label{n:comm0}
	&\gamma_{s}^{\pm\frac{1}2} \textrm{are central such that\ } \gamma_s^{\frac{1}{2}}\gamma_s^{-\frac{1}{2}}=1,\\\label{n:comm1}
	&\textrm{ $q^{\pm d}$ and $K_i^{\pm}$ commute with each other and\,} K_i^{\pm 1}K_i^{\mp 1}=1=q^{\pm d}q^{\mp d},\\
	\label{n:comm2}
	&K_ix_j^{\pm}(\underline{0})K_i^{-1}=q_i^{\pm a_{ij}}x_j^{\pm}(\underline{0}),\quad K_ix_{0}^{-\epsilon}(\epsilon e_s)K_i^{-1}=q_i^{-\epsilon a_{i0}}x_{0}^{-\epsilon}(\epsilon e_s),\\\label{n:comm3}
	&[\,x_i^{\epsilon}(\underline{0}),\,x_{0}^{-\epsilon}(\epsilon e_s)\,]=0,
		\quad \hbox{for} \quad i\neq 0,\\\label{n:comm4}
	&[\,x_i^{+}(\underline{0}),\,x_j^{-}(\underline{0})\,]=\delta_{ij}\frac{K_i-K_i^{-1}}{q_i-q_i^{-1}},\qquad [\,x_{0}^{+}(-e_s),\,x_{0}^{-}(e_{s})\,]=\frac{\gamma_s^{-1}K_{0}-\gamma_s K_{0}^{-1}}{q_0-q_0^{-1}},\\\label{n:comm5}
&q^{d}x^{\pm}_i(\underline{0})q^{-d}=q^{\pm\delta_{0,i}}x^{\pm}_i(\underline{0}),\quad q^{d}x^{-\epsilon}_{0}(\epsilon e_s)q^{-d}=q^{-\epsilon}x^{-\epsilon}_{0}(\epsilon e_s),\\ \label{n:comm6}	
&[\,x_{0}^{-\epsilon}(\epsilon e_s),\,x_{0}^{-\epsilon}(\underline{0})\,]_{q_0^{-2}}=0,
	\qquad [\,x_{0}^{-\epsilon}(\epsilon e_s),\,x_{0}^{-\epsilon}(\epsilon e_{s'})\,]=0, \hbox{for}\, s\neq s'\in J,\\ \label{n:comm7}
		&\sum\limits_{t=0}^{\ell=1-a_{ij}}(-1)^{t}\Big[{\ell\atop
		t}\Big]_i(x_i^{\pm}(\underline{0}))^{\ell-t}x_j^{\pm}(\underline{0})(x_i^{\pm}(\underline{0}))^{t}=0, \qquad a_{ij}\leqslant 0,\\ \label{n:comm8}
	&\sum\limits_{t=0}^{\ell=1-a_{j0}}(-1)^{t}\Big[{\ell\atop
		t}\Big]_j(x_j^{\epsilon}(\underline{0}))^{\ell-t}x_{0}^{\epsilon}(-\epsilon e_s)(x_j^{\epsilon}(\underline{0}))^{t}=0, \qquad\mbox{ $a_{{0}j}\leqslant 0$}, \\ \label{n:comm9}
&\sum\limits_{t=0}^{\ell=1-a_{0j}}(-1)^{t}\Big[{\ell\atop
		t}\Big]_0(x_0^{\epsilon}(-\epsilon e_s))^{\ell-t}x_{j}^{\epsilon}(\underline{0})x_0^{\epsilon}(-\epsilon e_s))^{t}=0, \qquad\mbox{ $a_{{0}j}\leqslant 0$},\\
\label{n:comm10}	
&\rm{Sym}_{\underline{0},{-\epsilon e_s}}\sum\limits_{t=0}^{\ell=1-a_{0j}}(-1)^{t}\Big[{\ell\atop
		t}\Big]_0(x_0^{\epsilon}(-\epsilon e_s))^{\ell-t}x_{j}^{\epsilon}(\underline{k})(x_0^{\epsilon}(\underline{0}))^{t}=0, \quad\mbox{ $a_{{0}j}\leqslant 0$ and $\underline{k}=\underline{0}$},\\
\label{n:comm11}
&\sum\limits_{t=0}^{3}(-1)^{t}q_0^{-2t}\Big[{3\atop
		t}\Big]_0(x_0^{\epsilon}(\underline{0}))^{3-t}x_{0}^{-\epsilon}(\epsilon e_s)(x_0^{\epsilon}(\underline{0}))^{t}=0, \qquad\mbox{ for $s\in J$},\\
\label{n:comm12}
&\sum\limits_{t=0}^{3}(-1)^{t}q_0^{2 t}\Big[{3\atop
		t}\Big]_0(x_0^{\epsilon}(-\epsilon e_s))^{3-t}x_{0}^{-\epsilon}(\underline{0})(x_0^{\epsilon}(-\epsilon e_s))^{t}=0, \qquad\mbox{ for $s\in J$},
\end{align}
\begin{align}
\label{n:comm13}
&\sum\limits_{t=0}^{3}(-1)^{t}\Big[{3\atop
		t}\Big]_0(x_0^{\epsilon}(-\epsilon e_s))^{3-t}x_{0}^{-\epsilon}(\epsilon e_{s'})(x_0^{\epsilon}(-\epsilon e_s))^{t}=0,
\qquad\mbox{ for $s\neq s'\in J$},
\end{align}
where $\rm{Sym}_{{m_1}, {m_{2}}}$  denotes the symmetrization with respect to the indices $({m_1}, {m_{2}})$.	
\end{defi}

\begin{remark}\, Notice that all the relations above are part of those in Definition \ref{defi-ntor}, as they only involve with
special modes or generators. Though relations \eqref{n:comm11} and \eqref{n:comm12} seem new,  they can be deduced from relations \eqref{n:Dr7} by \eqref{n:Dr8} and \eqref{n:Dr5} respectively. { In fact, relations \eqref{n:comm11} and \eqref{n:comm12} are the following two relations in
Definition \ref{defi-ntor}, respectively,
\begin{equation*}
\begin{split}
&[x_0^{\epsilon}(\underline{0}), x_0^{\epsilon}(\epsilon e_s)]_{q_0^{-2}}=0,\\
&[x_0^{\epsilon}(-\epsilon e_s), x_0^{\epsilon}(-2\epsilon e_s)]_{q_0^{2}}=0.
\end{split}
\end{equation*}}
\end{remark}
\smallskip

From the definition of $\mathcal{U}_0$(${\mathfrak g}_{N,tor}$), we have the following proposition.

\begin{prop}\label{N-symmetry}\, For
	$s\in J$, the following map $\tau_s$ defines an automorphism of $\mathcal{U}_0$(${\mathfrak g}_{N,tor})$:
$$\tau_s(x_{0}^{\epsilon}(-\epsilon e_{s'}))=\left\{\begin{array}{cl} (q^d)^{(-1-\epsilon)d_0}x_0^{\epsilon}(\underline{0})(q^d)^{(1-\epsilon)d_0},
	&  \textit{ if\ $s=s'$;}\\
x_{0}^{\epsilon}(-\epsilon e_{s'}),
	&   \textit{if\ $s\neq s'$},
	\end{array}\right.~~~~ \tau_s(\gamma_{s'})=\left\{\begin{array}{cl} \gamma_{s}^{-1},
	&  \textit{ if\ $s=s'$;}\\
	\gamma_{s'}\gamma_s^{-1},
	&   \textit{if\ $s\neq s'$},
	\end{array}\right.
 $$
	$$\tau_s(x_0^{\epsilon}(\underline{0}))=(q^{d})^{(\epsilon+1)d_0}x_{0}^{\epsilon}(-\epsilon e_s)(q^{d})^{(\epsilon-1)d_0}, \quad \tau_s(x_i^{\pm}(\underline{0}))=x_i^{\pm}(\underline{0}),$$
$$\tau_s(K_0)=\gamma_s^{-1}K_0, \quad
	\tau_s(K_i)=K_i, \quad \tau_s(q^d)=q^d,$$	
where $i=1, 2, \cdots, n, s'\in J $ and $\epsilon=\pm$ or $\pm1$.
\end{prop}

\begin{proof}\, It suffices to check that $\tau_s$ keeps the relations \eqref{n:comm0}-\eqref{n:comm13}. First we verify the relations \eqref{n:comm4}, and note that $[\,x_i^{+}(\underline{0}),\,x_j^{-}(\underline{0})\,]
=\delta_{ij}\frac{K_i-K_i^{-1}}{q_i-q_i^{-1}}$ follows from the definition of $\tau_s$. By definition
\begin{equation*}
\begin{split}
&\tau_s([\,x_{0}^{+}(-e_s),\,x_{0}^{-}(e_{s})\,])=[\,(q^d)^{-2d_0}x_{0}^{+}(\underline{0}),\,x_{0}^{-}(\underline{0})(q^d)^{2d_0}\,]
=\frac{K_0-K_0^{-1}}{q_0-q_0^{-1}},
\end{split}
\end{equation*}
which matches with $\tau_s(\frac{\gamma_s^{-1}K_{0}-\gamma_s K_{0}^{-1}}{q_0-q_0^{-1}})$ due to $\tau_s(K_0)=\gamma_s^{-1}K_0$. Similarly
the equality is checked for $s\neq s'$.

To check relations \eqref{n:comm6}, by definition it follows that
\begin{equation*}
\begin{split}
&\tau_s(
[\,x_{0}^{+}(-e_s),\,x_{0}^{+}(\underline{0})\,]_{q_0^{-2}})=[\,(q^d)^{-2d_0}x_{0}^{+}(\underline{0}),\,(q^d)^{2d_0}x_{0}^{+}(-e_s)\,]_{q_0^{-2}}\\
&=q_0^{-2}[\,x_{0}^{+}(\underline{0}),\,x_{0}^{+}(-e_s)\,]_{q_0^{2}}=0.
\end{split}
\end{equation*}
Similarly,
\begin{equation*}
\tau_s(
[\,x_{0}^{+}(-e_s),\,x_{0}^{+}(-e_{s'})\,])
=(q^d)^{-2d_0}[\,x_{0}^{+}(\underline{0}),\,x_{0}^{+}(-e_{s'})\,]_{q_0^{2}}=0.
\end{equation*}

For relation \eqref{n:comm13}, one has that
\begin{equation*}
\begin{split}
&\tau_s(\sum\limits_{t=0}^{3}(-1)^{t}\Big[{3\atop
		t}\Big]_0(x_0^{+}(-e_s))^{3-t}x_{0}^{-}(e_{s'})(x_0^{+}(-e_s))^{t})\\
&=(q^d)^{-6d_0}\sum\limits_{t=0}^{3}(-1)^{t}q_0^{-2t}\Big[{3\atop
		t}\Big]_0(x_0^{+}(\underline{0}))^{3-t}x_{0}^{-}(e_{s'})(x_0^{+}(\underline{0}))^{t}=0.
\end{split}
\end{equation*}
We can verify  that $\tau_s$ keeps the other relations in the same way.
\end{proof}

Proposition \ref{N-symmetry} reveals a symmetry of the algebra $\mathcal{U}_0$(${\mathfrak g}_{N,tor}$), which will be shown by the Dynkin diagram in the next subsection.

\subsection{$\mathcal{U}_0$(${\mathfrak g}_{N,tor}$) and the extended quantized GIM algebra of $N$-fold affinizations}
In this subsection, we focus on showing that the algebra $\mathcal{U}_0$(${\mathfrak g}_{N,tor}$) can be realized as a quotient of the extended quantized GIM algebra of  $N$-fold affinization.

First let us denote the following elements of $\mathcal{U}_0({\mathfrak g}_{N,tor})$: for $i\in I$ and $s\in J$,
\begin{align*}
F_{-s}&=x_{0}^{-}(e_s)(q^{-d})^{2d_0},\quad E_{-s}=(q^{d})^{2d_0}x_{0}^{+}(-e_s), \quad K_{-s}=\gamma_s^{-1} K_0,\\
E_i&=x_i^+(\underline{0}), \quad F_i=x_i^-(\underline{0}),\quad K_i=K_i,
\end{align*}
and set
$d_s=d_0$ for $s\in \{-N+1, \cdots, -1\}$ and $q_i=q^{d_i}$ for $i\in \tilde{J}$.

\begin{remark}\, Note that  $M$ given in Definition \ref{defi3.10} is a symmetrizable GIM of $N$-fold affinization of $A$: $D_MM$ is symmetric for
the diagonal matrix $D_M=\sum\limits_{i\in \tilde{J}}d_iE_{ii}=\left(\begin{array}{lr}
d_0I_N & 0\\
0 & D_0
\end{array}\right)$ where $D_0=\diag(d_i|i\in I_0)$, $d_i$ was given in Section 3.2.
\end{remark}

 Based on the aforementioned elements we have the following result. 

\begin{prop} The algebra $\mathcal{U}_0({\mathfrak g}_{N,tor})$ is the associative algebra over $\mathbb{K}$ generated by $E_i, F_i, K_i, q^{\pm d} (i\in \tilde{J})$,	satisfying the following relations:
\begin{align}
&\begin{array}{ll}\label{3.34}
	& K_i^{\pm1}K_i^{\mp1}= q^{\pm d}q^{\mp d}=1, \quad \textrm{$q^{\pm d}$ and $K_i^{\pm}$ commute with each other},\\
	&K_iE_jK_i^{-1}=q_i^{m_{ij}}E_j,\qquad K_iF_jK_i^{-1}=q_i^{-m_{ij}}F_j,\\
	&[\,E_i,\,F_j\,]=\delta_{i,j}\frac{K_i-K_i^{-1}}{q_i-q_i^{-1}};
	\end{array}\\
&\begin{array}{ll}\label{3.35}
	&q^dE_iq^{-d}=q E_i, \qquad q^dF_iq^{-d}=q^{-1} F_i, \qquad\quad i\in \{-N+1, \cdots, -1, 0\},\\
	&q^dE_jq^{-d}=E_j, \qquad q^dF_jq^{-d}=F_j, \qquad\qquad\quad j \in \{1, 2, \cdots, n\};
		\end{array}\\
&	\begin{array}{ll}\label{3.36}
	&[\,E_i, F_j\,]=0, \\
	&\sum\limits_{s=0}^{1-m_{ij}}(-1)^{s}\Big[{1-m_{ij}\atop
		s}\Big]_i
	E_{i}^{1-m_{ij}-s}E_jE_{i}^{s}=0,\\
	&\sum\limits_{s=0}^{1-m_{ij}}(-1)^{s}\Big[{1-m_{ij}\atop
		s}\Big]_i
	F_{i}^{1-m_{ij}-s}F_jF_{i}^{s}=0,
	\end{array}  \qquad i\neq j\in \tilde{J},  m_{ij}<0;
\end{align}
\begin{align}
& \begin{array}{ll}\label{3.37}
	&[\,E_i, E_j\,]=0=[F_i, F_j],\\
	&\sum\limits_{s=0}^{1+m_{ij}}(-1)^{s}\Big[{1+m_{ij}\atop
		s}\Big]_i
	E_{i}^{1+m_{ij}-s}F_jE_{i}^{s}=0, \\
	&\sum\limits_{s=0}^{1+m_{ij}}(-1)^{s}\Big[{1+m_{ij}\atop
		s}\Big]_i
	F_{i}^{1+m_{ij}-s}E_jF_{i}^{s}=0,
	\end{array} \qquad i\neq j\in \tilde{J}, m_{ij}>0;\\
\label{3.38}
&\quad	[\,E_i, E_j\,]=0=[E_i, F_j]=[F_i, F_j], \qquad i\neq j\in \tilde{J}, m_{ij}=0; \\
&	\begin{array}{ll}\label{3.39}
	&[\,E_{-j},\, [\,E_0,\,E_i\,]_{q_0}]_{q_0}+
	[\,E_0,\, [\,E_{-j},\,E_i\,]_{q_0}]_{q_0^{-3}}=0,\\
	&[\,F_{-j},\, [\,F_0,\,F_i\,]_{q_0^{-1}}]_{q_0^{3}}+
	[\,F_0,\, [\,F_{-j},\,F_i\,]_{q_0^{-1}}]_{q_0^{-1}}=0,
	\end{array}  \quad m_{i0}=1, 0\neq i\in I, j\in J;
\end{align}
	\begin{equation}
	\begin{array}{ll}\label{3.40}
	&[\,E_{-j},\,[\,E_0,\, [\,E_0,\,E_i\,]_{q_0^2}]_{q_0^{-2}}]_{q_0^{4}}+
	[[\,E_0,\, [\,E_{-j},\,[\,E_0,\, E_i\,]_{q_0^2}]]_{1}]_{q_0^{-2}}\vspace{6pt}\\
	&\hspace{4.82cm} +
	[[\,E_{0},\, [\,E_{0},\,[\,E_{-j},\, E_i\,]_{q_0^2}]]_{q_0^{-4}}]_{q_0^{-2}}=0,\vspace{9pt}\\
	&[\,F_{-j},\,[\,F_0,\, [\,F_0,\,F_i\,]_{q_0^{-2}}]_{q_0^{4}}]_{q_0^{2}}+
	[[\,E_0,\, [\,E_{-j},\,[\,E_0,\, E_i\,]_{q_0^{-2}}]]_{1}]_{q_0^{2}}\vspace{6pt}\\
	&\hspace{4.68cm}+
	[[\,E_{0},\, [\,E_{0},\,[\,E_{-j},\, E_i\,]_{q_0^{-2}}]]_{q_0^{2}}]_{q_0^{-4}}=0,\vspace{9pt}\\
	&[\,E_{-j},\,[\,E_{-j},\, [\,E_0,\,E_i\,]_{q_0^2}]_{1}]_{q_0^{2}}+
	[[\,E_{-j},\, [\,E_{0},\,[\,E_{-j},\, E_i\,]_{q_0^2}]]_{q_0^{-4}}]_{q_0^{2}}\vspace{6pt}\\
	&\hspace{4.7cm} +
	[[\,E_{0},\, [\,E_{-j},\,[\,E_{-j},\, E_i\,]_{q_0^2}]]_{1}]_{q_0^{2}}=0,\vspace{9pt}\\
	&[\,F_{-j},\,[\,F_{-j},\, [\,F_0,\,F_i\,]_{q_0^{-2}}]_{1}]_{q_0^{-2}}+
	[[\,F_{-j},\, [\,F_{0},\,[\,F_{-j},\, F_i\,]_{q_0^{-2}}]]_{q_0^{4}}]_{q_0^{-2}}\vspace{6pt}\\
	&\hspace{4.92cm} +
	[[\,F_{0},\, [\,F_{-j},\,[\,F_{-j},\, F_i\,]_{q_0^{-2}}]]_{1}]_{q_0^{-2}}=0,
	\end{array}
	\end{equation}
where   $m_{ij}\in M$ defined in Definition \ref{defi3.10}.
\end{prop}

\begin{proof}\, First of all, we remark that all generators $E_i, F_i, K_i, q^d$ are simply rewriting of the generators
of $\mathcal{U}_0({\mathfrak g}_{N,tor})$, therefore the isomorphism follows by listing the corresponding
 relations. In fact, relation \eqref{3.34} holds by relations \eqref{n:comm0}-\eqref{n:comm4}. It is easy to see that relation \eqref{3.35} follows from relation \eqref{n:comm5} and relation \eqref{3.36} holds from relations \eqref{n:comm6}-\eqref{n:comm7}. For relation \eqref{3.37}, it suffices to check the relations involving with $m_{ij}=2$ for $i\neq j$, since the
other relations can be verified directly. Specifically we need to show that for $i\neq j\in\{-N+1, \cdots, -2, -1, 0\}$.
$$\sum\limits_{s=0}^{1+m_{ij}}(-1)^{s}\Big[{3\atop
		s}\Big]_i
	E_{i}^{1+m_{ij}-s}F_jE_{i}^{s}=0.$$
We will check it for two cases: $m_{-s\, -s'}=2$ for $s\neq s'\in J$, $m_{0\, -s}=2$ and $m_{-s\, 0}=2$ for $s\in J$.
It follows from relation \eqref{n:comm13} for the first case. For the last two cases, it holds from \eqref{n:comm11} and \eqref{n:comm12}.

For relation \eqref{3.38}, we consider the case of $m_{0\, -s}=0$ for example, that is,
\begin{equation*}
	\begin{array}{ll}
	&[\,E_{0},\,E_{-s}]=[\,x_0^+(0),\,(q^{d})^{2d_0}x_{0}^{+}(-e_{s})]=q^{-2}(q^{d})^{2d_0}[\,x_0^+(0),\,x_{0}^{+}(-e_{s})]_{q^2}=0,
\end{array}
\end{equation*}
which follows from relation \eqref{n:comm13} by using \eqref{n:comm6}.	

Relations \eqref{3.39}-\eqref{3.40} hold by the Serre relations \eqref{n:comm7}-\eqref{n:comm12} directly.
\end{proof}

\begin{remark}\, From Proposition \ref{N-symmetry}, there exists an automorphism $\tau_{\sigma}$ of the algebra ${{\mathcal U}}_0({\mathfrak g}_{N,tor})$ for $\sigma\in S_{X}$ where $X=\{0, -1, \cdots, -N+1\}$, such that $\tau_{\sigma}(q^d)=q^d$, $\tau_{\sigma}(\gamma_s)=\gamma_{-\sigma(-i)}$ for $s\in J$ and $i\in K$,

$$\tau_{\sigma}(E_j)=\left\{\begin{array}{ll} E_{\sigma{(j)}},
	&  \textit{ if }~~j\in X;\vspace{12pt}\\
	E_{j},
	&   \textit{if}~~~ j\notin X,
	\end{array}\right.\qquad
\tau_{\sigma}(F_j)=\left\{\begin{array}{ll} F_{\sigma{(j)}},
	&  \textit{ if }~~~ j\in X;\vspace{12pt}\\
	F_{j},
	&   \textit{if} ~~ j\notin X,
	\end{array}\right.$$
$$\tau_{\sigma}(K_j)=\left\{\begin{array}{ll} K_{\sigma{(j)}},
	&  \textit{ if }~~~ j\in X;\vspace{12pt}\\
	K_{j},
	&   \textit{if}~~~ j\notin X,
	\end{array}\right.$$
where $\gamma_0$ is defined in Remark \ref{remark3.4}.
\end{remark}

Therefore we have the following Corollary immediately.
\begin{coro}\,\label{maincoro}
	The algebra  ${{\mathcal U}}_0({\mathfrak g}_{N,tor})$  is isomorphic to the quotient algebra $U_{q}(\mathcal{L}(M))/K$ of the extended quantized GIM algebra of the $N$-fold affinization $U_{q}(\mathcal{L}(M))$. That is,
\begin{gather*}
{\mathcal U}_{0}(\mathfrak{g}_{N,tor})\cong U_{q}(\mathcal{L}(M))/K,
\end{gather*}
where $K$ is the ideal of  $U_{q}(\mathcal{L}(M))$ generated by Serre relations \eqref{3.39} and \eqref{3.40}.
\end{coro}

Moreover, we have the following two results, whose proofs will be given in
the following two sections for the cases of  $N=2$ and $N\geq 3$, respectively.
 \begin{theo}\label{maintheo2-1}
	\, As an associative algebra, the quantum $2$-toroidal algebra
	$U_{q}({\mathfrak g}_{2,tor})$  is isomorphic to a quotient algebra of $\mathcal{U}_0$(${\mathfrak g}_{2,tor})$ for type $A$ and   $\mathcal{U}_0$(${\mathfrak g}_{2,tor})$ itself  for other types. More specifically, one has 
	
\begin{gather*}
U_{q}(\mathfrak {g}_{2,tor})\cong\left\{\begin{array}{ll} \mathcal{U}_0(\mathfrak{g}_{2,tor})/H_1,
	&  \textit{ for  type A;}\vspace{12pt}\\
	\mathcal{U}_0(\mathfrak{g}_{2,tor}),
	&   \textit{ otherwise},
	\end{array}\right.
\end{gather*}
where 
$H_1$ is an ideal defined in Section 4. 
\end{theo}

 \begin{theo}\label{maintheo2-2}
	\, As an associative algebra, the algebra
	$\overline{U}_{q}({\mathfrak g}_{N,tor})$ ($N>2$)  is isomorphic to a quotient algebra of $\mathcal{U}_0$(${\mathfrak g}_{N,tor})$:
\begin{gather*}
\overline{U}_{q}(\mathfrak {g}_{N,tor})\cong \mathcal{U}_0(\mathfrak{g}_{N,tor})/H_2,
	\end{gather*}
where the ideal $H_2$ will be defined in Section 5. 

\end{theo}

Combining Theorem \ref{maintheo2-1} and Theorem \ref{maintheo2-2}  with Corollary \ref{maincoro}, we obtain the following  main theorem, which generalizes a well-known result of Berman and Moody for Lie algebras \cite{BM}.
\begin{theo}
	\, The  algebra
	$\overline{U}_{q}({\mathfrak g}_{N,tor})$ are isomorphic to quotient algebras of the extended quantized GIM algebras of $N$-fold affinization $U_{q}(\mathcal{L}(M))$. 
%
\end{theo}

\section{Proof of Theorem \ref{maintheo2-1}}\label{S:case1}

\subsection{The algebra $\mathcal{U}_0$($\mathfrak{g}_{2,tor}$)}\, Recall that the algebra $\mathcal{U}_0$($\mathfrak{g}_{2, tor}$)
was defined in Definition \ref{def3.5} for $N=2$ and $J=\{1\}$. To simplify notation, for $i\in I$ and
$k\in \mathbb{Z}$, we denote that $x_i^{\pm}(0)\doteq x_i^{\pm}(\underline{0})$,  $x_i^{\pm}(k)\doteq x_i^{\pm}(ke_1)$
and $\gamma^{\pm \frac{1}{2}}\doteq \gamma_1^{\pm \frac{1}{2}}$.

The following result is immediate by definition. 

\begin{prop}
There is a $\mathbb{C}$-algebra anti-involution $\iota$ of $\mathcal{U}_0$($\mathfrak{g}_{2,tor}$) such that $\iota:  x_i^{\pm}(k)\mapsto  x_i^{\mp}(-k)$,
$K_i\mapsto K_i^{-1}$,
$q^{d}\mapsto q^{-d}$, $\gamma^{\frac{1}{2}}\mapsto \gamma^{-\frac{1}{2}}$ and $q\mapsto q^{-1}$. 
\end{prop}

Before proving Theorem \ref{maintheo2-1}, we recall a useful result, which can be checked easily (cf. \cite[Lemma 4.1]{JZ1}).

\begin{lemm}\label{N:iso} Suppose the associative algebras $\mathcal A=\langle x_i\rangle/(R_1) $ and $\mathcal B=\langle \hat x_i, y_j\rangle/(\hat R_1, R_2, R_{12})$ with respective relations $R_1=R_1(x_i), \hat R_1=R_1(\hat x_i), R_2=R_2(y_j), R_{12}=R_{12}(\hat x_i, y_j)$. Define the map $\phi: \mathcal A=\langle x_i\rangle/(R_1) \longrightarrow \mathcal B=\langle \hat x_i, y_j\rangle/(\hat R_1, R_2, R_{12})$ such that $ x_i\mapsto \hat x_i$.
	If $y_j\in Im\phi$ inside $\mathcal B$, and $R_2(y_j)\subset (\hat R_1)$, $R_{12}(\hat x_i, y_j)\subset (\hat R_1)$,
	then $\mathcal A\simeq \mathcal B$ as associative algebras.
\end{lemm}

By this lemma, verification is made easy by only checking the relations \eqref{n:comm0}-\eqref{n:comm13} for the set of simplified generators
of the  quantum $2$-toroidal algebra displayed in the algebra $\mathcal{U}$($\mathfrak{g}_{2,tor}$).

\subsection{Proof of Theorem \ref{maintheo2-1}}\,

We now set out to prove Theorem \ref{maintheo2-1}. The idea is to show that a simplified set of relations are satisfied for the quantum
algebra $\mathcal{U}_0(\mathfrak{g}_{2,tor})$ in the toroidal case. Some of the computations are made in \cite{JZ1} for the affine type $A$,
so we will focus on the relations special for the toroidal case.
The proof is divided into four steps:

{\bf Step 1:} We will prove that $\mathcal U_0(\mathfrak{g}_{2,tor})$ contains all other elements $x_0^{\pm}(\epsilon k)$, $a_0(\epsilon k)$ involving the index $i=0$
in Definition \ref{defi-ntor}, these elements satisfy relations consistent with Definition \ref{defi-ntor}. Indeed, there exists an subalgebra of $\mathcal U_0(\mathfrak{g}_{2,tor})$ generated by $x_0^{\pm}(\epsilon k)$, $a_0(\epsilon k), \gamma^{\frac{1}{2}}$ and $K_0^{\pm 1}$, and we denote it by $U_q(\widehat{\mathfrak{sl}}_2)_0$.

Here we use induction on degree $k\in \mathbb{N}$, first we introduce the following useful elements for $k=1$ in $\mathcal U_0(\mathfrak{g}_{2,tor})$:
\begin{gather}\label{4.1}
a_0(1)=K_0^{-1}\gamma^{1/2}\,\bigl[\,x_0^+(0),\,x_0^-(1)\,\bigr]\in
{\mathcal U}_0(\mathfrak{g}_{2,tor}),\\\label{4.2}
a_0(-1)=K_0{\gamma}^{-1/2}\,\bigl[\,x_0^+(-1),\,x_0^-(0)\,
\bigr]\in
{\mathcal U}_0(\mathfrak{g}_{2,tor}),
\end{gather}
which are used to inductively generate higher degree elements using a spiral argument
based on Lemma \ref{N:iso}.



Furthermore, for $\epsilon=\pm$ or $\pm1$ we have that
\begin{gather}\label{4.3}
x_0^{\epsilon}(\epsilon)=[2]_0^{-1}\gamma^{\frac{\epsilon}{2}}\,\bigl[\,a_0(\epsilon),\,x_0^{\epsilon}(0)\,\bigr]\in
{\mathcal U}_0(\mathfrak{g}_{2,tor}).
\end{gather}

Now we check the relation \eqref{n:Dr7} involving with the elements ${x}_0^{\epsilon}(0)$ and ${x}_0^{\epsilon}(\epsilon)$.
\begin{prop} Using the above notations, we have that
\begin{align}\label{4.4}
&[{x}_0^{\epsilon}(0),\, {x}_0^{\epsilon}(\epsilon)]_{q_0^{-2}}=0.
\end{align}
\end{prop}
\begin{proof}\, We only check in the case of $\epsilon=+$ or $+1$, it is similar for the case of $\epsilon=-$ or $-1$. It follows from \eqref{4.3} and \eqref{4.1} that
\begin{eqnarray*}
[{x}_0^{+}(0),\, {x}_0^{+}(1)]_{q_0^{-2}}&=&\Big[x_0^{+}(0), [2]_0^{-1}\gamma^{\frac{1}{2}}\,\bigl[\,a_0(1),\,x_0^{+}(0)\,\bigr]\Big]\\
&=&-q_0^{4}[2]^{-1}\gamma K_0^{-1}\Bigl[x_0^{+}(0), \bigl[\,x_0^{+}(0),\, \bigl[\,x_0^+(0),\,x_0^-(1)\,\bigr]\bigr]_{q_0^{-2}}\Bigr]_{q_0^{-4}}=0,
\end{eqnarray*}
where we have used the relation \eqref{n:comm11}. \end{proof}


The following proposition gives key relations among the degree-$1$ elements $K_0^{\pm 1}$, $x_0^{\epsilon}(0)$, $a_0(\epsilon)$, $x_0^{\epsilon}(\epsilon)$ and $x_0^{\epsilon}(-\epsilon)$, which
are consistent with Definition \ref{defi-ntor}.

\begin{prop}\label{p1} Using the above notations, we have the following relations (as above and below $\epsilon=\pm$ or $\pm 1$):
\begin{align}\label{4.5}
&\bigl[\,{a}_0(\epsilon),\,{x}_0^{\epsilon}(-\epsilon)\,\bigr]=\epsilon [2]_0\gamma^{-\frac{\epsilon}{2}} {x}_0^{\epsilon}(0),\quad \bigl[\,{a}_0(-\epsilon),\,{x}_0^{\epsilon}(0)\,\bigr]=\epsilon [2]_0\gamma^{-\frac{\epsilon}{2}}{x}_0^{\epsilon}(-\epsilon),\\\label{4.6}
&\bigl[\,{a}_0(1), {a}_0(-1)\,\bigr]=[2]_0\frac{\gamma-\gamma^{-1}}{q_0-q_0^{-1}},\\\label{4.7}
&\bigl[\,{x}_0^+(\epsilon),\,{x}_0^-(0)\,\bigr]=\gamma \bigl[\,{x}_0^+(0),\,{x}_0^-(\epsilon)\,\bigr],\\\label{4.8}
&{a}_0(\epsilon)=\epsilon K_0^{-\epsilon}\gamma^{\frac{1}{2}\epsilon}\bigl[\,{x}_0^{\epsilon}(0),\,\,{x}_0^{-\epsilon}(\epsilon )\bigr]= \epsilon K_0^{-\epsilon}\gamma^{-\frac{1}{2}\epsilon} \bigl[\,{x}_0^{\epsilon}(\epsilon), {x}_0^{-\epsilon}(0)\bigr],\\\label{4.9}
&[\,{x}_0^{\pm}(1),\,{x}_0^{\pm}(-1)\,]_{q_0^{\pm2}}+[\,{x}_0^{\pm}(0),\,{x}_0^{\pm}(0)\,]_{q_0^{\pm2}}=0,\\\label{4.10}
&\bigl[\,{a}_0(-\epsilon),\,{x}_0^{\epsilon}(\epsilon)\,\bigr]=\epsilon [2]_0\gamma^{-\frac{\epsilon}{2}} {x}_0^{\epsilon}(0),\\\label{4.11}
&[\,{x}_0^+(1),\,{x}_0^-(-1)\,]=\frac{\gamma K_0-\gamma^{-1} K_0^{-1}}{q_0-q_0^{-1}}.
\end{align}
\end{prop}
\begin{proof} \, Most of the relations can be easily checked as in \cite{JZ1} for type A. Here we check relation \eqref{4.9} for the case of $+$, others can be verified by the construction directly.
Note that $[\,x_0^{+}(0),\,x_0^{+}(-1)\,]_{q^{2}}=0$ by \eqref{n:comm6}, then
\begin{equation*}
\begin{split}
0&=\Big[\,a_0(1),\,[\,x_0^{+}(0),\,x_0^{+}(-1)\,]_{q^{2}}\,\Big]\\
&=\Bigl(\Big[\,[a_0(1),\,x_0^{+}(0)],\,x_0^{+}(-1)\,\Big]_{q^{2}}+
\Big[\,x_0^{+}(0),\,[a_0(1),\,x_0^{+}(-1)\,]\,\Big]_{q^{2}}\Big)\\
&=[2]\gamma^{-\frac{1}{2}}\Big([\,x_0^{+}(1),\,x_0^{+}(-1)\,]_{q^{2}}
+[\,x_0^{+}(0),\,x_0^{+}(0)\,]_{q^{2}}\Big).
\end{split}
\end{equation*}
It means that $[\,x_0^{+}(1),\,x_0^{+}(-1)\,]_{q^{2}}
+[\,x_0^{+}(0),\,x_0^{+}(0)\,]_{q^{2}}=0$, which is consistent with the defining relation in Definition \ref{defi-ntor}.
\end{proof}

Now we construct all degree-$k$ elements $x_0^{\pm}(k),\,x_0^{\pm}(-k),\,a_0(\pm k)$ involving with index $i=0$ by inductively
 as follows.  For $\epsilon=\pm$ or $\pm1$, we set that
\begin{gather}\label{4.12}
x_{0}^{\pm}(\epsilon k)=\pm [2]_0^{-1}\gamma^{\pm\frac{1}{2}}\,\bigl[\,a_0(\epsilon),\,x_{0}^{\pm}(\epsilon (k-1))\,\bigr]\in
{\mathcal U}_0(\mathfrak{g}_{2,tor}),\\\label{4.13}
{\phi}_{0}(k)=(q_0-q_0^{-1})\gamma^{\frac{2-k}{2}}\,\bigl[\,x_{0}^+(k-1),\,x_{0}^{-}(1)\,\bigr]\in
{\mathcal U}_0(\mathfrak{g}_{2,tor}),\\\label{4.14}
{\varphi}_{0}(-k)=-(q_0-q_0^{-1})\gamma^{\frac{k-2}{2}}\,\bigl[\,x_{0}^+(-1),\,x_{0}^{-}(-k+1)\,\bigr]\in
{\mathcal U}_0(\mathfrak{g}_{2,tor}),
\end{gather}
where $a_0(\pm k)$ are defined by ${\phi}_0(k)$ and ${\varphi}_0(-k)\, (k\geq 0)$ 
as follows:
\begin{gather*}\sum\limits_{m=0}^{\infty}{\phi}_0(r) z^{-r}=K_0 \exp \Big(
(q_0{-}q_0^{-1})\sum\limits_{\ell=1}^{\infty}
 a_0(\ell)z^{-\ell}\Big), \\
\sum\limits_{r=0}^{\infty}{\varphi}_0(-r) z^{r}=K_0^{-1}\exp
\Big({-}(q_0{-}q_0^{-1})
\sum\limits_{\ell=1}^{\infty}a_0(-\ell)z^{\ell}\Big).
\end{gather*}

A partition of $k$, denoted $\lambda\vdash k$, is a non-increasing sequence of positive integers $\lambda_1\geq\lambda_2\geq\cdots\geq \lambda_l>0 $ such that $\lambda_1+\lambda_2+\cdots+\lambda_l=k$, where $l(\lambda)=l$ is called the number of parts.
A partition $\lambda=(\lambda_1\lambda_2\cdots)$ can also be denoted as $(1^{m_1}2^{m_2}\cdots)$ with multiplicity of $i$ being $m_i$.
Then we obtain the following formulas between $a_0(\pm k)$ and  ${\phi}_0(\pm k)$: 
 \begin{align}
 \label{4.15}
 {\phi}_0(k)&=K_0\sum\limits_{\lambda\vdash k}\frac{(q_0-q_0^{-1})^{l(\lambda)}}{m_{\lambda}!} a_0(\lambda),
  \\ \label{4.16}
 {\phi}_0(-k)&=K_0^{-1}\sum\limits_{\lambda\vdash k}\frac{(q_0^{-1}-q_0)^{l(\lambda)}}{m_{\lambda}!} a_0(-\lambda),
\end{align}
where $m_{\lambda}!=\prod_{i\geq 1}m_i!$
and $ a_0(\pm\lambda)= a_0(\pm\lambda_1)a_0(\pm\lambda_2)\cdots.$
It is easy to see $\iota( \phi_0(k))=\varphi_0(-k)$.

By the inductive hypothesis on degree, assume that all degree-$n$ elements for $n\in \mathbb{N}$ satisfy the relevant relations in Definition \ref{defi-ntor}, now we are left to check that so do all degree-$(n+1)$ elements for $n\in \mathbb{N}$. The proof proceeds in the following propositions (cf. \cite[Sect. 4]{JZ1} for type $A$). 

\begin{prop}\label{p10} From the above construction, we have the following relations for $\epsilon=\pm$ or $\pm1$, $-n+1\leqslant l_1\leqslant n-1$ and $-n\leqslant l_2\leqslant n$,
\begin{align}\label{4.23}
&K_i{x}_{0}^{\pm}(\epsilon(n+1))K_i^{-1}=q_i^{a_{i0}}{x}_{0}^{\pm}(\epsilon(n+1)),\\\label{4.24}
&q^{d}{x}_{0}^{\pm}(\epsilon(n+1))q^{-d}=q^{\epsilon(n+1)} {x}_{0}^{\pm}(\epsilon(n+1)),\\\label{4.25}
&[\,{x}_{0}^{\pm}(n+1),\,{x}_{0}^{\pm}(n-1)\,]_{q_0^{\pm2}}+[\,{x}_{0}^{\pm}(n),\,{x}_{0}^{\pm}(n)\,]_{q_0^{\pm2}}=0,\\\label{4.26}
&[\,{x}_{0}^{\pm}(n+2),\,{x}_{0}^{\pm}(n-1)\,]_{q_0^{\pm2}}+[\,{x}_{0}^{\pm}(n),\,{x}_{0}^{\pm}(n+1)\,]_{q_0^{\pm2}}=0,\\
\label{4.27}
&[\,{x}_{0}^{\pm}(n+1),\,{x}_{0}^{\pm}(n)\,]_{q_0^{\pm2}}=0,\\
\label{4.28}
&[\,{x}_0^{\pm}(n+1),\,{x}_0^{\pm}(l_1)\,]_{q_0^{\pm2}}+[\,{x}_{0}^{\pm}(l_1+1),\,{x}_{0}^{\pm}(n)\,]_{q_0^{\pm2}}=0,\\
\label{4.29}
&[\,{x}_{0}^+(n+1),\,{x}_{0}^-(l_2)\,]=\frac{\gamma^{\frac{n+1-l}{2}}\phi_0(n+l_2+1)}{q_0-q_0^{-1}},\\\label{4.30}
&[\,{x}_0^+(n+1),\,{x}_0^-(-n-1)\,]=\frac{\gamma^{n+1} K_0-\gamma^{-n-1} K_0^{-1}}{q_0-q_0^{-1}}.
\end{align}
\end{prop}

\begin{proof} We only check  the ``+'' case, as the ``-'' case can be obtained similarly. \eqref{4.23} and \eqref{4.24} hold directly by the construction.
The proof of relation \eqref{4.25} is similar to that of \eqref{4.26}. Let's consider 
\eqref{4.26}, by the inductive hypothesis,
 take the bracket of $a_0(2)$ and $[\,x_0^{+}(n),\,x_0^{+}(n-1)\,]_{q^{2}}=0$, then
\begin{equation*}
\begin{split}
0&=\Big[\,a_0(2),\,[\,x_0^{+}(n),\,x_0^{+}(n-1)\,]_{q_0^{2}}\,\Big]\\
&=\Bigl(\Big[\,[a_0(2),\,x_0^{+}(n)],\,x_0^{+}(n-1)\,\Big]_{q_0^{2}}+
\Big[\,x_0^{+}(n),\,[a_0(2),\,x_0^{+}(n-1)\,]\,\Big]_{q_0^{2}}\Big)\\
&=\frac{[4]_0}{2}\gamma^{-1}\Big([\,x_0^{+}(n+2),\,x_0^{+}(n-1)\,]_{q_0^{2}}
+[\,x_0^{+}(n),\,x_0^{+}(n+1)\,]_{q_0^{2}}\Big),
\end{split}
\end{equation*}
which implies $$[\,x_0^{+}(n+2),\,x_0^{+}(n-1)\,]_{q_0^{2}}
+[\,x_0^{+}(n),\,x_0^{+}(n+1)\,]_{q_0^{2}}=0.$$

In a similar manner, we can prove relations
\eqref{4.27} and \eqref{4.28}.
Note that by \eqref{4.25}$$[\,x_0^{+}(n+1),\,x_0^{+}(n-1)\,]_{q_0^{2}}
+[\,x_0^{+}(n),\,x_0^{+}(n)\,]_{q_0^{2}}=0.$$  Therefore, it is easy to see that
\begin{equation*}
\begin{split}
0&=\Big[\,a_0(1),\,[\,x_0^{+}(n+1),\,x_0^{+}(n-1)\,]_{q_0^{2}}
+[\,x_0^{+}(n),\,x_0^{+}(n)\,]_{q_0^{2}}\,\Big]\\
&=[2]_0\gamma^{-\frac{1}{2}}\Bigl([x_0^{+}(n+2),x_0^{+}(n-1)]_{q_0^{2}}+
[x_0^{+}(n+1), x_0^{+}(n)]_{q_0^{2}}\\
&\hskip1.85cm+[x_0^{+}(n+1), x_0^{+}(n)]_{q_0^{2}}
+[x_0^{+}(n), x_0^{+}(n+1)]_{q_0^{2}}\Bigr)\\
&=2[2]_0\gamma^{-\frac{1}{2}}[\,x_0^{+}(n+1),\,x_0^{+}(n)\,]_{q_0^{2}},\\
\end{split}
\end{equation*}
where we have used \eqref{4.26}.  It yields that $[\,x_0^{+}(n+1),\,x_0^{+}(n)\,]_{q_0^{2}}=0$.

%
%

In order to check \eqref{4.29}, we consider that for $-n\leq l_2\leq n$,
\begin{align*}
&[\,x_{0}^+(n+1),\,x_0^-(l_2)\,]\\
&=[2]_0^{-1}\gamma^{\frac{1}{2}}\,\Bigl(\Big[\,\bigl[\,a_0(1),\,x_{0}^-(l_2)\,\bigr],\,
x_{0}^+(n)\,\Big]
+\Big[\,a_0(1),\,\bigl[\,x_{0}^+(n),\,x_{0}^-(l_2)\,\bigr]\,\Big]\,\Bigr)\\
&=-\gamma\,\bigl[\,x_0^-(l_2+1),\,x_{0}^+(n)\,\bigr]=\frac{\gamma^{\frac{n+1-l_2}{2}}\phi_0(n+l_2+1)}{q_0-q_0^{-1}}.
\end{align*}

 For \eqref{4.30}, one has that
\begin{equation*}
\begin{split}
&[\,x_0^+(n+1),\,x_0^-(-n-1)\,]\\
&=-[2]_0^{-2}\Bigl(\Big[\bigl[a_0(1), [x_0^+(n), a_0(-1)] \bigr],\,x_0^-(-n)\Big]
+\Bigl[a_0(-1), \bigl[[a_0(1), x_0^-(-n)], x_0^+(n)\,\bigr]\Bigr]\Bigr)\\
&=\frac{\gamma^{n+1} K_0-\gamma^{-n-1} K_0^{-1}}{q_0-q_0^{-1}}.
\end{split}
\end{equation*}
\end{proof}

\begin{prop} \label{p11'} The following relations hold for $d=\gamma^{-\frac{1}{2}}q_0^2$,
\begin{align}\label{4.31}
&[\,\bar{{\phi}}_0(r),\,{x}_{0}^{+}(m)\,]\\\nonumber
&\hskip0.5cm=[2]_0\gamma^{-\frac{1}{2}}\Bigl(\sum\limits_{t=1}^{r-1}(q_0-q_0^{-1}){d}^{t-1}{x}_{0}^{+}(m+t)\bar{{\phi}}_0(r-t)
+d^{r-1}{x}_{0}^{+}(r+m)\Bigr),\\\label{4.32}
&[\,\bar{{\varphi}}_0(-r),\,{x}_{0}^{-}(-m)\,]\\\nonumber
&\hskip0.5cm=-[2]_0\gamma^{-\frac{1}{2}}\Bigl(\sum\limits_{t=1}^{r-1}(q_0-q_0^{-1}){d}^{t-1}\bar{{\varphi}}_0(-r+t){x}_{0}^{-}(-m-t)
+d^{r-1}{x}_{0}^{-}(-r-m)\Bigr).
\end{align}
\end{prop}

\begin{proof} Here we only check relation \eqref{4.31} for the case of $m=-1$, other cases are similar.
\begin{align*}
&[\,\bar{\phi}_0(n+1),\,x_{0}^{+}(-1)\,]\\
&=\gamma^{\frac{1-n}{2}}K_0^{-1}\Bigl(\bigl[\,x_0^+(n),[x_0^-(1),\,x_{0}^{+}(-1)]\bigr]_{q_0^{2}}+
\bigl[[\,x_0^+(n),\,x_{0}^{+}(-1)]_{q_0^{2}},x_0^-(1)\bigr]\Bigr)\\
&=d\,\Bigl([\,\bar{\phi}_0(n),\,x_0^+(0)]+(q_0-q_0^{-1})q_0^{-2}[2]_0x_0^+(0)\bar{\phi}_0(n)\Bigl)=\cdots\\
&=d^n\,\Bigl([\,\bar{\phi}_0(1),\,x_0^+(n-1)]+(q_0-q_0^{-1})q_0^{-2}[2]_0x_0^+(n-1)\bar{\phi}_0(1)\Bigl)\\
&\hskip1cm+\sum\limits_{t=1}^{n-1}{d}^t(q_0-q_0^{-1})q_0^{-2}[2]_0x_0^+(t-1)\bar{\phi}_0(n+1-t)\\
&=[2]_0\gamma^{-\frac{1}{2}}\sum\limits_{t=1}^{n}{d}^{t-1}(q_0-q_0^{-1})x_0^+(t-1)\bar{\phi}_0(n+1-t)+d^nx_0^+(n).
\end{align*}
\end{proof}

\begin{prop} \label{p11} The following relations hold for $\epsilon=\pm$ or $\pm1$,
\begin{align}\label{4.33}
&[\,{a}_{0}(\epsilon(n+1)),\,{x}_{0}^{\epsilon}(-\epsilon)\,]=\frac{[2(n+1)]_0}{n+1}\gamma^{-\frac{n+1}{2}}{x}_{0}^{\epsilon}(\epsilon n),\\\label{4.34}
&[\,{a}_{0}(\epsilon(n+1)),\,{x}_{0}^{-\epsilon}(-\epsilon n)\,]=\frac{[2(n+1)]_0}{n+1}\gamma^{-\frac{n+1}{2}}{x}_{0}^{-\epsilon}(\epsilon),\\\label{4.35}
&[\,{a}_{0}(n+1),\,{a}_0(-n-1)\,]=\frac{[2(n+1)]_0}{n+1}\frac{\gamma^{n+1}-\gamma^{-(n+1)}}{q_0-q_0^{-1}}.
\end{align}
\end{prop}
\begin{proof} 
Note that \eqref{4.34} is similar to \eqref{4.33}, so
we only deal with \eqref{4.33}. It follows from Proposition \ref{p11'} that
\begin{equation*}
\begin{split}
&[\,a_{0}(n+1),\,x_{0}^{+}(-1)\,]\\
&=\bigl[\bar{\phi}_0(n+1),\,x_{0}^{+}(-1)\bigr]-\frac{(q-q^{-1})}{n+1}\sum\limits_{t=1}^{n}t\bigl[\bar{\phi}_0(n+1-t)a_0(t),\,x_{0}^{+}(-1)\bigr]\\
&=\frac{[2(n+1)]_0}{(n+1)!}\gamma^{-\frac{n+1}{2}}x_{0}^{+}(n).
\end{split}
\end{equation*}

Notice that by \eqref{4.15},
\begin{equation*}
\begin{split}
&-\gamma^{-\frac{1}{2}}\frac{K_0^{-1}}{q_0-q_0^{-1}}[\,x_0^+(0),\phi_0(n)]_{q_0^{2}}=
-\gamma^{-\frac{1}{2}}q_0^{-2}
\sum\limits_{\substack{\lambda\vdash n}}\frac{(q_0-q_0^{-1})^{l(\lambda)-1}}{m_{\lambda}!}\bigl[\,x_0^+(0),\,a_0(\lambda)\bigr]_{q_0^{4}}
\end{split}
\end{equation*}
As a consequence, one has that
\begin{equation*}
\begin{split}
&[\,a_{0}(n+1),\,x_{0}^{+}(-1)\,]\\
&=-\gamma^{-\frac{1}{2}}q_0^{-2}
\sum\limits_{t=1}^{n}\frac{(q_0-q_0^{-1})^{t-1}}{t!}\sum \limits_{\substack{ 1\leqslant i_1,\cdots,i_t\leqslant n-t+1,\\
  i_1+i_2+\cdots+i_t=n}} [\,x_0^+(0), a_0(i_1)\cdots a_0(i_t)]_{q^{4}}\\
&\hskip0.55cm-\sum\limits_{t=2}^{n+1}\frac{(q_0-q_0^{-1})^{t-1}}{t!}\sum \limits_{\substack{ 1\leqslant i_1,\cdots,i_t\leqslant n-t+2,\\
  i_1+i_2+\cdots+i_{t}=n+1}}\Bigl[a_0(i_1)\cdots a_0(i_{t}),\,x_{0}^{+}(-1)\Bigr]=\frac{[2(n+1)]_0}{(n+1)!}\gamma^{-\frac{n+1}{2}}x_{0}^{+}(n).
\end{split}
\end{equation*}

For \eqref{4.35}, it follows from \eqref{4.33}-\eqref{4.34} and \eqref{4.15} by the inductive hypothesis that
\begin{align*}
&[\,a_{0}(n+1),\,a_{0}(-n-1)\,]\\
&=\gamma^{\frac{n-1}{2}}K_0[\,a_{0}(n+1),\,[\,x_0^+(-1),\,x_0^-(-n)\,]]-\sum\limits_{\substack{\lambda\vdash n+1\\ \lambda\neq(n+1)}}\frac{(q_0-q_0^{-1})^{l(\lambda)-1}}{m_{\lambda}!}\bigl[\,a_{0}(n+1),\,a_0(\lambda)\bigr]\\
&=\gamma^{\frac{n-1}{2}}K_0\bigl([\,[a_{0}(n+1),\,x_0^+(-1)],\,x_0^-(-n)\,]
+[\,x_0^+(-1),\,[a_{0}(n+1),\,x_0^-(-n)\,]\bigr)\\
&=\frac{[2(n+1)]_0}{n+1}\frac{\gamma^{n+1}-\gamma^{-(n+1)}}{q_0-q_0^{-1}}.
\end{align*}
\end{proof}

So far, we have shown that the algebra $\mathcal{U}_0(\mathfrak{g}_{2, tor})$  contains the subalgebra $U_q(\hat{sl_2})_0$,
which is isomorphic to the quantum affine algebra for type $\hat{sl_2}$.


{\bf Step 2:} We will construct generators ${x}_1^{+}(-1)$ and ${x}_1^{-}(1)$ in $\mathcal{U}_0(\mathfrak{g}_{2, tor})$ and prove that ${x}_1^{+}(-1)$ and ${x}_1^{-}(1)$
satisfy the same relations as those of ${x}_0^{+}(-1)$, ${x}_0^{-}(1)$ in Definition \ref{def3.5}. Furthermore, we can  construct another subalgebra $U_q(\hat{sl_2})_1$ generated by the node-1 elements such as $x_1^{\pm}(\epsilon k)$, $a_1(k)$, $x_1^{\pm}(0)$, $K_1^{\pm}$ and $\gamma^{\frac{1}{2}}$ by repeating step 1. Moreover, we will check the relations between $U_q(\hat{sl_2})_0$ and $U_q(\hat{sl_2})_1$.

For $\epsilon=\pm1$ or $\pm$, we define that
\begin{gather}\label{4.36}
{x}_1^{\pm}(\epsilon)=\pm\gamma^{\pm\frac{1}{2}}\,\bigl[\,{a}_0(\epsilon),\,{x}_1^{\pm}(0)\,\bigr]\in
\mathcal{U}_0(\mathfrak{g}_{2, tor}),
\end{gather}

Using this construction, one has that
\begin{align}
&K_i{x}_1^{\pm}(\epsilon)K_i^{-1}=q_i^{\pm a_{i1}}{x}_1^{\pm}(1),\quad
q^{d}{x}_1^{\pm}(\epsilon)q^{-d}=q^{\pm 1} {x}_1^{\pm}(1),\\
&[\,{x}_1^{\epsilon}(-\epsilon),\,{x}_1^{\epsilon}(0)\,]_{q_1^{-2}}=0,\\
&[\,{x}_i^{-\epsilon}(0),\,{x}_1^{\epsilon}(-\epsilon)\,]=0, \ \ \mbox{for} \ \ i\neq 1, \\
&[\,{x}_1^+(-1),\,{x}_1^-(1)\,]=\frac{\gamma^{-1} K_1-\gamma K_1^{-1}}{q_1-q_1^{-1}}.
\end{align}

To verify the Serre relations involving $x_1^{\epsilon}(-\epsilon)$ for $\epsilon=\pm$ or $\pm 1$, we need to do some preparation.
Similar to step 1, we construct
\begin{gather}\label{4.37}
{a}_1(1)=\gamma^{\frac{1}{2}}K_1^{-1}\,\bigl[\,{x}_{1}^+(0),\,{x}_1^{-}(1)\,\bigr]\in
\mathcal{U}_0(\mathfrak{g}_{2, tor}),\\\label{4.38}
{a}_1(-1)=\gamma^{-\frac{1}{2}}K_1\,\bigl[\,{x}_{1}^+(-1),\,{x}_1^{-}(0)\,\bigr]\in
\mathcal{U}_0(\mathfrak{g}_{2, tor}).
\end{gather}

Similar to Proposition \ref{p1}, we have the following relations.

\begin{prop}\label{l5} It is easy to see that for $\epsilon=\pm 1$,
\begin{align}\label{4.39}
&\bigl[\,{x}_1^+(\epsilon),\,{x}_1^-(0)\,\bigr]=\gamma \bigl[\,{x}_1^+(0),\,{x}_1^-(\epsilon)\,\bigr],\\\label{4.40}
&\bigl[\,{a}_1(1), {a}_1(-1)\,\bigr]=[2]_1\frac{\gamma-\gamma^{-1}}{q_1-q_1^{-1}},\\\label{4.41}
&{a}_1(1)=K_1^{-1}\gamma^{\frac{1}{2}}\bigl[\,{x}_1^+(0),\,{x}_1^-(1)\,\bigr]=K_1^{-1}\gamma^{-\frac{1}{2}} \bigl[\,{x}_1^+(1),\,{x}_1^-(0)\,\bigr],\\
\label{4.42}
&{a}_1(-1)=K_1\gamma^{-\frac{1}{2}}\bigl[\,{x}_1^+(-1),\,{x}_1^-(0)\,\bigr]=K_1\gamma^{\frac{1}{2}} \bigl[\,{x}_1^+(0),\,{x}_1^-(-1)\,\bigr].
\end{align}
\end{prop}

Now we proceed to check that $x_1^{\epsilon}(-\epsilon)$ keeps the Serre relations involving $x_0^{\epsilon}(-\epsilon)$ in Definition \ref{def3.5} (c.f. \eqref{n:comm10}-\eqref{n:comm12}).

\begin{prop}\label{p13} From the above construction, we have the following relations, which are consistent with the defining relations of $U_q(\mathfrak{g}_{2, tor})$ for $\epsilon=\pm$ or $\pm1$.
\begin{align}\label{4.43}
&\bigl[\,{x}_0^{\epsilon}(0),\,[{x}_0^{\epsilon}(0),\,{x}_1^{\epsilon}(-\epsilon)]_{q_0^{-1}}\bigr]_{q_0}=0,\\\label{4.44}
&\bigl[\,{x}_1^{\epsilon}(-\epsilon),\,[{x}_1^{\epsilon}(0),\,{x}_0^{\epsilon}(0)]_{q_1}\bigr]_{q_1^{-1}}
+\bigl[\,{x}_1^{\epsilon}(0),\,[{x}_1^{\epsilon}(-\epsilon),\,{x}_0^{\epsilon}(0)]_{q_1^{-1}}\bigr]_{q_1}=0,\\\label{4.45}
&\bigl[\,{x}_1^{\epsilon}(0),\,[{x}_1^{\epsilon}(0),\,[{x}_1^{\epsilon}(0),\,{x}_1^{-\epsilon}(\epsilon)]_1]_{q_1^{-2}}\bigr]_{q_1^{-4}}=0.
\end{align}
\end{prop}
\begin{proof}\, To check \eqref{4.43}, by using \eqref{4.36} together with \eqref{v1}, we have that for $\epsilon=-1$ ,
\begin{eqnarray*}
&&\bigl[\,{x}_0^{-}(0),\,[{x}_0^{-}(0),\,{x}_1^{-}(1)]_{q_0^{-1}}\bigr]_{q_0}\\
&=&\gamma^{-\frac{1}{2}}\bigl[\,{x}_0^{-}(0),\,[{x}_0^{-}(0),\,[{a}_0(1),\,{x}_1^{-}(0)]]_{q_0^{-1}}\bigr]_{q_0}\\
&=&\gamma^{-\frac{1}{2}}\Bigl(\bigl[\,{x}_0^{-}(0),\,[[{x}_0^{-}(0),\,{a}_0(1)],\,{x}_1^{-}(0)]_{q_0^{-1}}\bigr]_{q_0}
+\bigl[\,{x}_0^{-}(0),\,[{a}_0(1), [{x}_0^{-}(0),\,{x}_1^{-}(0)]_{q_0^{-1}}]\bigr]_{q_0}\Bigr)\\
&=&[2]_0\bigl[\,{x}_0^{-}(0),\,[{x}_0^{-}(1),\,{x}_1^{-}(0)]_{q_0^{-1}}\bigr]_{q_0}+[2]_0\bigl[\,{x}_0^{-}(1), [{x}_0^{-}(0),\,{x}_1^{-}(0)]_{q_0^{-1}}\bigr]_{q_0}\\
&&+\gamma^{-\frac{1}{2}}\bigl[\,{a}_0(1),[{x}_0^{-}(0),\, [{x}_0^{-}(0),\,{x}_1^{-}(0)]_{q_0^{-1}}]_{q_0}\bigr]\\
&=&0,
\end{eqnarray*}
where the last equality uses the relations \eqref{n:comm10} and \eqref{n:comm7}.

By using the $q$-bracket,
the left hand side of \eqref{4.44} for the case of $\epsilon=-$ can be seen as follows.
\begin{eqnarray*}
&&\hbox{LHS of \eqref{4.44}}=\gamma^{-\frac{1}{2}}\Big[a_0(1), \big[x_1^{-}(0), [x_1^{-}(0), x_{0}^{-}(0)]_{q_1}\big]_{q_1^{-1}}\Big].
\end{eqnarray*}
Thus \eqref{4.44} follows from the Serre relation $\big[x_1^{-}(0), [x_1^{-}(0), x_{0}^{-}(0)]_{q_1}\big]_{q_1^{-1}}=0$.

In fact, relation \eqref{4.45} holds since it is equivalent to $[\,{x}_1^{\epsilon}(0),\,{x}_1^{\epsilon}(\epsilon)]_{q_1^{-2}}=0$.
\end{proof}

Next we turn to check the inter-relations between subalgebras $U_q(\hat{sl_2})_0$ and $U_q(\hat{sl_2})_1$.

\begin{prop}\label{p13} From the above construction, we have that for $\epsilon=\pm$ or $\pm1$
\begin{align}\label{b34}
&\bigl[\,{a}_0(1), {a}_1(-1)\,\bigr]=[a_{01}]_0\frac{\gamma-\gamma^{-1}}{q_1-q_1^{-1}},\\\label{b36}
&[\,{x}_1^{\pm}(1),\,{x}_0^{\pm}(0)\,]_{q_1^{\pm a_{10}}}+[\,{x}_0^{\pm}(1),\,{x}_1^{\pm}(0)\,]_{q_1^{\pm a_{10}}}=0,
\\\label{b37}
&[\,{a}_0(\epsilon),\,{x}_1^{\pm}(-\epsilon)\,]=[a_{01}]_0\gamma^{\mp\frac{1}{2}}{x}_1^{\pm}(0),\\
\label{b39}
&\bigl[\,{x}_0^{\pm}(0),\,[{x}_0^{\pm}(0),\,{x}_1^{\pm}(1)]_{q_0^{-1}}\bigr]_{q_0}=0.
\end{align}
\end{prop}

%
%

To complete Step 2, we have to check the Serre relations for non-simply laced cases. Without loss of generality, it is sufficient to show the Serre relations for type $C_n$. 
 \begin{prop}\label{lemm3.1} In the case of type $C_n$ for $\epsilon_i=0$ or $1$, it holds that
 	\begin{align}\label{identity3.1}
 	&\rm{Sym}_{\epsilon_1,\epsilon_2, \epsilon_3}\sum\limits_{s=0}^{3}(-1)^{s}\Big[{3\atop
 		s}\Big]_1x_1^{-}(\epsilon_1)\cdots x_1^{-}(\epsilon_s)x_0^{-}(0)x_1^{-}(\epsilon_{s+1})\cdots x_1^{-}(\epsilon_{3})=0.
 	\end{align}
 \end{prop}
 \begin{proof}\, The proof is divided into four cases. \\
 	
 	$(1)$\, The case of $\epsilon_1=\epsilon_2=\epsilon_3=0$ is trivial. \\
 	
 	$(2)$\,  For the case of $\epsilon_1=1$ and $\epsilon_2=\epsilon_3=0$. Note that by
 	the Serre relation in terms of the simple generators (\ref{n:comm7}-\ref{n:comm8}) (using q-brackets), one has that
 	\begin{align*}
 	&A_1=[\,x_1^-(0), [\,x_1^-(0),\, [\,x_1^{-}(0),\,x_0^{-}(0)\,]_{q_1^{2}}]_{q_1^{-2}}]_1=0,\\
 	&B_1=[\,x_1^-(0), [\,x_1^-(0),\, [\,x_1^{-}(0),\,x_0^{-}(1)\,]_{q_1^{2}}]_{q_1^{-2}}]_1=0.
 	\end{align*}
Therefore, $[a_0(1),\, A_1]=0$ implies that
 	\begin{align*}
 	&[\,x_1^-(1), [\,x_1^-(0),\, [\,x_1^{-}(0),\,x_0^{-}(0)\,]_{q_1^{2}}]_{q_1^{-2}}]_1+[\,x_1^-(0), [\,x_1^-(1),\, [\,x_1^{-}(0),\,x_0^{-}(0)\,]_{q_1^{2}}]_{q_1^{-2}}]_1\\
 	&+[\,x_1^-(0), [\,x_1^-(0),\, [\,x_1^{-}(1),\,x_0^{-}(0)\,]_{q_1^{2}}]_{q_1^{-2}}]_1-[2][\,x_1^-(0), [\,x_1^-(0),\, [\,x_1^{-}(1),\,x_0^{-}(1)\,]_{q_1^{2}}]_{q_1^{-2}}]_1=0.
 	\end{align*}
 	Notice that the last summand is killed by $B_1=0$. So we obtain that 
 	\begin{equation*}
 	\begin{split}
 	&C_1\doteq[\,x_1^-(1), [\,x_1^-(0),\, [\,x_1^{-}(0),\,x_0^{-}(0)\,]_{q_1^{2}}]_{q_1^{-2}}]_1+[\,x_1^-(0), [\,x_1^-(1),\, [\,x_1^{-}(0),\,x_0^{-}(0)\,]_{q_1^{2}}]_{q_1^{-2}}]_1\\
 	&\hskip2cm+[\,x_1^-(0), [\,x_1^-(0),\, [\,x_1^{-}(1),\,x_0^{-}(0)\,]_{q_1^{2}}]_{q_1^{-2}}]_1=0.
 	\end{split}
 	\end{equation*}

 	$(3)$\, The case of $\epsilon_1=\epsilon_2=1$ and $\epsilon_3=0$. Similar to case $(2)$ and using $[a_0(2),\, A_1]=[a_0(1),\, B_1]=0$, we have that
 	\begin{equation*}
 	\begin{split}
 	0=&[a_0(1), C_1]\\
 	=&\gamma^{\frac{1}{2}}\Big([\,x_1^-(2), [\,x_1^-(0),\, [\,x_1^{-}(0),\,x_0^{-}(0)\,]_{q_1^{2}}]_{q_1^{-2}}]_1
 	\\
 	&\hskip1.05cm
 	+[\,x_1^-(1), [\,x_1^-(1),\, [\,x_1^{-}(0),\,x_0^{-}(0)\,]_{q_1^{2}}]_{q_1^{-2}}]_1\\
 	&\hskip2.05cm+[\,x_1^-(1), [\,x_1^-(0),\, [\,x_1^{-}(1),\,x_0^{-}(0)\,]_{q_1^{2}}]_{q_1^{-2}}]_1
 	\\
 	&\hskip2.85cm
 	-[2][\,x_1^-(1), [\,x_1^-(0),\, [\,x_1^{-}(0),\,x_0^{-}(1)\,]_{q_1^{2}}]_{q_1^{-2}}]_1\Big)\\
 	&+\gamma^{\frac{1}{2}}\Big([\,x_1^-(1), [\,x_1^-(1),\, [\,x_1^{-}(0),\,x_0^{-}(0)\,]_{q_1^{2}}]_{q_1^{-2}}]_1
 	\\
 	&\hskip1.05cm
 	+[\,x_1^-(0), [\,x_1^-(2),\, [\,x_1^{-}(0),\,x_0^{-}(0)\,]_{q_1^{2}}]_{q_1^{-2}}]_1\\
 	&\hskip2.05cm+[\,x_1^-(0), [\,x_1^-(1),\, [\,x_1^{-}(1),\,x_0^{-}(0)\,]_{q_1^{2}}]_{q_1^{-2}}]_1\\
 	&\hskip2.85cm
 	-[2][\,x_1^-(0), [\,x_1^-(1),\, [\,x_1^{-}(0),\,x_0^{-}(1)\,]_{q_1^{2}}]_{q_1^{-2}}]_1\Big)\\
 	&+\gamma^{\frac{1}{2}}\Big([\,x_1^-(1), [\,x_1^-(0),\, [\,x_1^{-}(1),\,x_0^{-}(0)\,]_{q_1^{2}}]_{q_1^{-2}}]_1\\
 	&\hskip1.05cm
 	+[\,x_1^-(0), [\,x_1^-(1),\, [\,x_1^{-}(1),\,x_0^{-}(0)\,]_{q_1^{2}}]_{q_1^{-2}}]_1\\
 	&\hskip2.05cm+[\,x_1^-(0), [\,x_1^-(0),\, [\,x_1^{-}(2),\,x_0^{-}(0)\,]_{q_1^{2}}]_{q_1^{-2}}]_1
 	\\
 	&\hskip2.85cm
 	-[2][\,x_1^-(0), [\,x_1^-(0),\, [\,x_1^{-}(1),\,x_0^{-}(1)\,]_{q_1^{2}}]_{q_1^{-2}}]_1\Big)\\
 	=&\gamma^{\frac{1}{2}}\Big([\,x_1^-(1), [\,x_1^-(1),\, [\,x_1^{-}(0),\,x_0^{-}(0)\,]_{q_1^{2}}]_{q_1^{-2}}]_1
 	\\
 	&\hskip1.05cm
 	+[\,x_1^-(1), [\,x_1^-(0),\, [\,x_1^{-}(1),\,x_0^{-}(0)\,]_{q_1^{2}}]_{q_1^{-2}}]_1\\
 	&\hskip2.05cm+[\,x_1^-(0), [\,x_1^-(1),\, [\,x_1^{-}(1),\,x_0^{-}(0)\,]_{q_1^{2}}]_{q_1^{-2}}]_1\Big)\\
 	&+\gamma^{\frac{1}{2}}\Big([\,x_1^-(2), [\,x_1^-(0),\, [\,x_1^{-}(0),\,x_0^{-}(0)\,]_{q_1^{2}}]_{q_1^{-2}}]_1
 	\\
 	&\hskip1.05cm
 	+[\,x_1^-(0), [\,x_1^-(2),\, [\,x_1^{-}(0),\,x_0^{-}(0)\,]_{q_1^{2}}]_{q_1^{-2}}]_1\\
 	&\hskip2.05cm+[\,x_1^-(0), [\,x_1^-(0),\, [\,x_1^{-}(2),\,x_0^{-}(0)\,]_{q_1^{2}}]_{q_1^{-2}}]_1\Big)\\
 	&-\gamma^{\frac{1}{2}}[2]\Big([\,x_1^-(1), [\,x_1^-(0),\, [\,x_1^{-}(0),\,x_0^{-}(1)\,]_{q_1^{2}}]_{q_1^{-2}}]_1
 \end{split}
 	\end{equation*}
 	\begin{equation*}
 	\begin{split}
 	&\hskip1.05cm
 	+[\,x_1^-(0), [\,x_1^-(1),\, [\,x_1^{-}(0),\,x_0^{-}(1)\,]_{q_1^{2}}]_{q_1^{-2}}]_1\\
 	&\hskip2.05cm+[\,x_1^-(0), [\,x_1^-(0),\, [\,x_1^{-}(1),\,x_0^{-}(1)\,]_{q_1^{2}}]_{q_1^{-2}}]_1\Big)\\
 	=&\gamma^{\frac{1}{2}}\Big([\,x_1^-(1), [\,x_1^-(1),\, [\,x_1^{-}(0),\,x_0^{-}(0)\,]_{q_1^{2}}]_{q_1^{-2}}]_1\\
 	&\hskip1.05cm+[\,x_1^-(1), [\,x_1^-(0),\, [\,x_1^{-}(1),\,x_0^{-}(0)\,]_{q_1^{2}}]_{q_1^{-2}}]_1\\
 	&\hskip2.05cm+[\,x_1^-(0), [\,x_1^-(1),\, [\,x_1^{-}(1),\,x_0^{-}(0)\,]_{q_1^{2}}]_{q_1^{-2}}]_1\Big),
 	\end{split}
 	\end{equation*}
 	which implies that $(4.16)$ holds for $\epsilon_1=\epsilon_2=1$ and $\epsilon_3=0$.
 	
 	$(4)$\,  The case of $\epsilon_1=\epsilon_2=\epsilon_3=1$ is checked similarly.
 	 	Therefore, Proposition \ref{lemm3.1} has been proved.\\

Now we can repeat step 1 to construct the generators involving the index $i=1$ as follows.

\begin{gather}\label{4.55}
x_{1}^{\pm}(\epsilon k)=\pm [2]_1^{-1}\gamma^{\pm\frac{1}{2}}\,\bigl[\,a_1(\epsilon),\,x_{1}^{\pm}(\epsilon (k-1))\,\bigr]\in
{\mathcal U}_0(\mathfrak{g}_{2,tor}),\\\label{4.56}
{\phi}_{1}(k)=(q_1-q_1^{-1})\gamma^{\frac{2-k}{2}}\,\bigl[\,x_{1}^+(k-1),\,x_{1}^{-}(1)\,\bigr]\in
{\mathcal U}_0(\mathfrak{g}_{2,tor}),\\\label{4.57}
{\varphi}_{1}(-k)=-(q_1-q_1^{-1})\gamma^{\frac{k-2}{2}}\,\bigl[\,x_{1}^+(-1),\,x_{1}^{-}(-k+1)\,\bigr]\in
{\mathcal U}_0(\mathfrak{g}_{2,tor}),
\end{gather}

The elements $x_{1}^{\pm}(\epsilon k)$ and $a_1(k)$ also satisfy the relevant relations consistent with Definition \ref{defi-ntor},  similar to Proposition \ref{p10} to Proposition \ref{p11}.

In the remaining part of Step 2, we will check the Serre relations on higher degree elements $x_{1}^{\pm}(\epsilon k)$ by induction on $k\in \mathbb{N}$.
By the inductive hypothesis, we assume that the Serre relations involving  $x_{1}^{\pm}(\epsilon m)$ for $m\leq n-1$ hold.  Then we have the following proposition.

\begin{prop}\, For $m, n, k\in \mathbb{N}$, the Serre relations holds

\begin{align}\label{4.58}
&\bigl[\,{x}_0^{\pm}(\epsilon k),\,[{x}_0^{\pm}(\epsilon k),\,{x}_1^{\pm}(\epsilon n)]_{q_0^{-1}}\bigr]_{q_0}=0,\\\label{4.59}
&\bigl[\,{x}_1^{\pm}(\epsilon k),\,[{x}_1^{\pm}(\epsilon k),\,{x}_0^{\pm}(\epsilon n)]_{q_1^{-1}}\bigr]_{q_1}=0,\\\label{4.60}
&\bigl[{x}_1^{\pm}(\epsilon n), [{x}_1^{\pm}(\epsilon (n+t)), {x}_0^{\pm}(\epsilon m)]_{q_1^{-1}}\bigr]_{q_1}+
\bigl[{x}_1^{\pm}(\epsilon (n+t)), [{x}_1^{\pm}(\epsilon n), {x}_1^{\pm}(\epsilon m)]_{q_1^{-1}}\bigr]_{q_1}=0.
\end{align}
\end{prop}

\begin{proof}\, Here we only check \eqref{4.58} for the case $-$ and $\epsilon=+$, it is similar for the other cases.
By \eqref{4.36} and \eqref{v1}, one has that
\begin{eqnarray*}
&&\bigl[\,{x}_0^{-}(k), [{x}_0^{-}(k),\,{x}_1^{-}(n)]_{q_0^{-1}}\bigr]_{q_0}\\
&=&\gamma^{-\frac{1}{2}}\bigl[\,{x}_0^{-}(k),\,[{x}_0^{-}(k),\,[{a}_0(1),\,{x}_1^{-}(n-1)]]_{q_0^{-1}}\bigr]_{q_0}\\
&=&\gamma^{-\frac{1}{2}}\bigl([{x}_0^{-}(k),\,[[{x}_0^{-}(k), {a}_0(1)], {x}_1^{-}(n-1)]_{q_0^{-1}}]_{q_0}
+[{x}_0^{-}(k), [{a}_0(1), [{x}_0^{-}(k), {x}_1^{-}(n-1)]_{q_0^{-1}}]]_{q_0}\bigr)\\
&=&[2]_1\bigl[\,{x}_0^{-}(k),\,[{x}_0^{-}(k+1),\,{x}_1^{-}(n-1)]_{q_0^{-1}}\bigr]_{q_0}+[2]_1\bigl[\,{x}_0^{-}(k+1), [{x}_0^{-}(k),\,{x}_1^{-}(n-1)]_{q_0^{-1}}\bigr]_{q_0}\\
&&+\gamma^{-\frac{1}{2}}\bigl[\,{a}_0(1),[{x}_0^{-}(k),\, [{x}_0^{-}(k),\,{x}_1^{-}(n-1)]_{q_0^{-1}}]_{q_0}\bigr]=0,
\end{eqnarray*}
where we have used the inductive hypothesis
$$[\,{x}_0^{-}(k),\,[{x}_0^{-}(k+1),\,{x}_1^{-}(n-1)]_{q_0^{-1}}]_{q_0}+[\,{x}_0^{-}(k+1), [{x}_0^{-}(k),\,{x}_1^{-}(n-1)]_{q_0^{-1}}]_{q_0}=0,$$
$$[{x}_0^{-}(k),\, [{x}_0^{-}(k),\,{x}_1^{-}(n-1)]_{q_0^{-1}}=0.$$

The proof of relation \eqref{4.59} is almost as same as that of relation \eqref{4.58}.

Similarly we only check relation \eqref{4.60} for the case $-$ and $\epsilon=+$, first it follows from \eqref{4.59} for all $m, n\in \mathbb{N}$

 $$A_2\doteq\bigl[\,{x}_1^{-}(n),\,[{x}_1^{-}(n),\,{x}_0^{-}(m)]_{q_1^{-1}}\bigr]_{q_1}=0.$$
Therefore it is clear that
\begin{equation*}
 	\begin{split}
 	0=&[a_1(t), A_2]=[a_1(t), [\,{x}_1^{-}(n),\,[{x}_1^{-}(n),\,{x}_0^{-}(m)]_{q_1^{-1}}]_{q_1}]\\
    =&-\frac{[2t]_1}{t}\gamma^{\frac{|t|}{2}}\bigl([\,{x}_1^{-}(n+t),\,[{x}_1^{-}(n),\,{x}_0^{-}(m)]_{q_1^{-1}}]_{q_1}+
    [\,{x}_1^{-}(n),\,[{x}_1^{-}(n+t),\,{x}_0^{-}(m)]_{q_1^{-1}}]_{q_1}\bigr) \\
    &+\frac{[t]_1}{t}\gamma^{\frac{|t|}{2}}[\,{x}_1^{-}(n),\,[{x}_1^{-}(n),\,{x}_0^{-}(m+t)]_{q_1^{-1}}]_{q_1},
 \end{split}
 	\end{equation*}
which implies that  $$[\,{x}_1^{-}(n+t),\,[{x}_1^{-}(n),\,{x}_0^{-}(m)]_{q_1^{-1}}]_{q_1}+
    [\,{x}_1^{-}(n),\,[{x}_1^{-}(n+t),\,{x}_0^{-}(m)]_{q_1^{-1}}]_{q_1}=0.$$

\end{proof}

{\bf Step 3:} Repeating the above two steps, we can obtain all elements $x_{i}^{\pm}(\epsilon k)$ and $a_{i}(\epsilon k)$  at the remaining nodes $1<i\in I$ similarly,
which satisfy all relevant relations in Definition \ref{defi-ntor}.

\begin{gather}\label{4.61}
{x}_i^{\pm}(\epsilon)=\pm\gamma^{\pm\frac{1}{2}}\,\bigl[\,{a}_{i-1}(\epsilon),\,{x}_i^{\pm}(0)\,\bigr]\in
\mathcal{U}_0(\mathfrak{g}_{2, tor}),\\\label{4.62}
{a}_i(1)=\gamma^{\frac{1}{2}}K_i^{-1}\,\bigl[\,{x}_{i}^+(0),\,{x}_i^{-}(1)\,\bigr]\in
\mathcal{U}_0(\mathfrak{g}_{2, tor}),\\\label{4.63}
{a}_i(-1)=\gamma^{-\frac{1}{2}}K_i\,\bigl[\,{x}_{i}^+(-1),\,{x}_i^{-}(0)\,\bigr]\in
\mathcal{U}_0(\mathfrak{g}_{2, tor}),\\\label{4.64}
x_{i}^{\pm}(\epsilon k)=\pm [2]_i^{-1}\gamma^{\pm\frac{1}{2}}\,\bigl[\,a_{i}(\epsilon),\,x_{i}^{\pm}(\epsilon (k-1))\,\bigr]\in
{\mathcal U}_0(\mathfrak{g}_{2,tor}),\\\label{4.65}
{\phi}_{i}(k)=(q_i-q_i^{-1})\gamma^{\frac{2-k}{2}}\,\bigl[\,x_{i}^+(k-1),\,x_{i}^{-}(1)\,\bigr]\in
{\mathcal U}_0(\mathfrak{g}_{2,tor}),\\\label{4.66}
{\varphi}_{i}(-k)=-(q_i-q_i^{-1})\gamma^{\frac{k-2}{2}}\,\bigl[\,x_{i}^+(-1),\,x_{i}^{-}(-k+1)\,\bigr]\in
{\mathcal U}_0(\mathfrak{g}_{2,tor}).
\end{gather}

Actually we can check all defining relations involving these elements in Definition \ref{defi-ntor} similar to the above two steps, except for the Serre relations of non-simply-laced cases. We verify the Serre relations of  type $B_n$ and  $G_2$ as follows.

 \begin{prop}\label{lemm3.2} For type $B_n$, the following Serre relation holds.
 	\begin{align}\label{identity3.2}
 	&\rm{Sym}_{\epsilon_1,\epsilon_2, \epsilon_3}\sum\limits_{s=0}^{3}(-1)^{s}\Big[{3\atop
 		s}\Big]_1x_n^{-}(\epsilon_1)\cdots x_n^{-}(\epsilon_s)x_{n-1}^{-}(0)x_n^{-}(\epsilon_{s+1})\cdots x_n^{-}(\epsilon_{3})=0.
 	\end{align}
 \end{prop}
 \begin{proof}\, It holds by induction on the index that
 \begin{align*}
 	&A_3=[\,x_n^-(0), [\,x_n^-(0),\, [\,x_n^{-}(0),\,x_{n-1}^{-}(0)\,]_{q_n^{2}}]_{q_n^{-2}}]_1=0,\\
 	&B_3=[\,x_n^-(0), [\,x_n^-(0),\, [\,x_n^{-}(0),\,x_{n-1}^{-}(1)\,]_{q_n^{2}}]_{q_n^{-2}}]_1=0.
 	\end{align*}
 the remaining proof is almost as the same as that of Proposition \ref{lemm3.1}.
 \end{proof}

 \begin{prop}\label{lemm3.3} For type $G_2$, we have that
 	\begin{align}\label{identity3.3}
 	&\rm{Sym}_{\epsilon_1,\cdots, \epsilon_4}\sum\limits_{s=0}^{4}(-1)^{s}\Big[{4\atop
 		s}\Big]_2x_2^{-}(\epsilon_1)\cdots x_2^{-}(\epsilon_s)x_{1}^{-}(0)x_2^{-}(\epsilon_{s+1})\cdots x_2^{-}(\epsilon_{4})=0.
 	\end{align}
 \end{prop}
 \begin{proof}\, The proof is divided into five cases according to the value of $\epsilon_i$. \\
 	
 	$(1)$\, The case of all $\epsilon_i=0$ is clear. \\
 	
 	$(2)$\,  For the case of $\epsilon_1=1$ and other $\epsilon_i=0$,
 	it follows from the Serre relation \eqref{n:comm7} that (using $q$-brackets),
 	\begin{equation*}
 	\begin{split}
 	&A_4=[\,x_2^-(0),[\,x_2^-(0), [\,x_2^-(0),\, [\,x_2^{-}(0),\,x_1^{-}(0)\,]_{q_1^{3}}]_{q_1^{-3}}]_{q_1}]_{q_1^{-1}}=0,\\
 	&B_4=[\,x_2^-(0),[\,x_2^-(0), [\,x_2^-(0),\, [\,x_2^{-}(0),\,x_1^{-}(1)\,]_{q_1^{3}}]_{q_1^{-3}}]_{q_1}]_{q_1^{-1}}=0.
 	\end{split}
 	\end{equation*}
 	Then $[a_1(1),\, A_4]=0$ and $B_4=0$ imply that
 	\begin{equation*}
 	\begin{split}
 	&C_4\doteq[\,x_2^-(1),[\,x_2^-(0), [\,x_2^-(0),\, [\,x_2^{-}(0),\,x_1^{-}(0)\,]_{q_1^{3}}]_{q_1^{-3}}]_{q_1}]_{q_1^{-1}}\\
 	&\hskip0.8cm+[\,x_2^-(0),[\,x_2^-(1), [\,x_2^-(0),\, [\,x_2^{-}(0),\,x_1^{-}(0)\,]_{q_1^{3}}]_{q_1^{-3}}]_{q_1}]_{q_1^{-1}}\\
 	&\hskip1.5cm+[\,x_2^-(0),[\,x_2^-(0), [\,x_2^-(1),\, [\,x_2^{-}(0),\,x_1^{-}(0)\,]_{q_1^{3}}]_{q_1^{-3}}]_{q_1}]_{q_1^{-1}}\\
 	&\hskip2.5cm+[\,x_2^-(0),[\,x_2^-(0), [\,x_2^-(0),\, [\,x_2^{-}(1),\,x_1^{-}(0)\,]_{q_1^{3}}]_{q_1^{-3}}]_{q_1}]_{q_1^{-1}}=0.
 	\end{split}
 	\end{equation*}\\
 	
 	$(3)$\,  For the case of $\epsilon_1=\epsilon_2=1$ and $\epsilon_3=\epsilon_4=0$, using $[a_1(2),\, A_4]=0$, $[a_1(1),\, B_4]=0$ and $[a_1(1),\, C_4]=0$ it follows that
 	\begin{equation*}
 	\begin{split}
 	D_4&\doteq[\,x_2^-(1),[\,x_2^-(1), [\,x_2^-(0),\, [\,x_2^{-}(0),\,x_1^{-}(0)\,]_{q_1^{3}}]_{q_1^{-3}}]_{q_1}]_{q_1^{-1}}\\
 	&\hskip0.6cm+[\,x_2^-(1),[\,x_2^-(0), [\,x_2^-(1),\, [\,x_2^{-}(0),\,x_1^{-}(0)\,]_{q_1^{3}}]_{q_1^{-3}}]_{q_1}]_{q_1^{-1}}\\
 	&\hskip1.2cm+[\,x_2^-(1),[\,x_2^-(0), [\,x_2^-(1),\, [\,x_2^{-}(1),\,x_1^{-}(0)\,]_{q_1^{3}}]_{q_1^{-3}}]_{q_1}]_{q_1^{-1}}\\
 	&\hskip1.8cm+[\,x_2^-(0),[\,x_2^-(1), [\,x_2^-(1),\, [\,x_2^{-}(0),\,x_1^{-}(0)\,]_{q_1^{3}}]_{q_1^{-3}}]_{q_1}]_{q_1^{-1}}\\
 	&\hskip2.4cm+[\,x_2^-(0),[\,x_2^-(1), [\,x_2^-(0),\, [\,x_2^{-}(1),\,x_1^{-}(0)\,]_{q_1^{3}}]_{q_1^{-3}}]_{q_1}]_{q_1^{-1}}\\
 	&\hskip3.0cm+[\,x_2^-(0),[\,x_2^-(0), [\,x_2^-(1),\, [\,x_2^{-}(1),\,x_1^{-}(0)\,]_{q_1^{3}}]_{q_1^{-3}}]_{q_1}]_{q_1^{-1}}=0.
 	\end{split}
 	\end{equation*}

 	$(4)$\,  For the case of $\epsilon_1=\epsilon_2=\epsilon_3=1$ and $\epsilon_4=0$, applying $[a_1(2),\, C_4]=0$ and $[a_1(1),\, B_4]=0$
 it follows that $[a_1(1),\, D_4]=0$.

 	$(5)$\,  The case of all $\epsilon_i=1$ can be checked similarly as $(4)$. 
 \end{proof}

{\bf Step 4:} There exists a natural homomorphism $\pi: \mathcal{U}_0(\mathfrak{g}_{2,tor})\rightarrow U_{q}(\mathfrak {g}_{2,tor})$ such that
$\pi(a)=a$  for $a\in \mathcal{U}_0(\mathfrak{g}_{2,tor})$, as we have been using the same notation for the
 elements in the former. From the three steps above, we have constructed all (or rather images of) the generators of $U_{q}(\mathfrak {g}_{2,tor})$ in $\mathcal{U}_0(\mathfrak{g}_{2,tor})$, and they keep all defining relations in Definition \ref{defi-ntor}, so $\pi$ is an epimorphism.
At the same time, it is clear that $\pi$ is injective except for $\mathfrak{g}$ being type $A$, since all relations of the latter can be derived
from or part of the former (except in type A) and the very same set of generators of $\mathcal{U}_0(\mathfrak{g}_{2,tor})$ does form a simplified set of generators for the latter (in the same notation).

In type $A$, note that the Dynkin diagram is cyclic, so for $\epsilon=\pm1$ or $\pm$, we can also define that
\begin{gather*}
y_n^{\pm}(\epsilon)=\pm \gamma^{\pm\frac{1}{2}}\,\bigl[\,a_0(\epsilon),\,x_n^{\pm}(0)\,\bigr],\\
b_n(1)=\gamma^{\frac{1}{2}}K_1^{-1}\,\bigl[\,x_{n}^+(0),\,y_n^{-}(1)\,\bigr],\\
b_n(-1)=\gamma^{-\frac{1}{2}}K_1\,\bigl[\,y_{n}^+(-1),\,x_n^{-}(0)\,\bigr].
\end{gather*}
We define the elements $y_i^{\pm}(\epsilon)$, $b_{i}(1)$ and $b_{i}(-1)$ for $i=n-1, n-2, , \cdots, 1, 0$ inductively by
\begin{gather*}
y_i^{\pm}(\epsilon)=\pm \gamma^{\pm\frac{1}{2}}\,\bigl[\,b_{i+1}(\epsilon),\,x_i^{\pm}(0)\,\bigr],\\
b_i(1)=\gamma^{\frac{1}{2}}K_1^{-1}\,\bigl[\,y_{i}^+(0),\,y_i^{-}(1)\,\bigr],\\
b_i(-1)=\gamma^{-\frac{1}{2}}K_1\,\bigl[\,y_{i}^+(-1),\,x_i^{-}(0)\,\bigr].
\end{gather*}
Arguing by the degree, one can define the higher degree elements $y_i^{\pm}(k)$ and $b_{i}(k)$ similar to $x_i^{\pm}(k)$ and $a_{i}(k)$ for $i\in I$ and $k\in \mathbb{Z}^*$ (c.f. Step 3).
Therefore $$\ker \pi=\{y_i^{\pm}(k)-x_i^{\pm}(k), b_{i}(k)-a_i(k)|i\in I, k\in \mathbb{Z}^*\}.$$

It is easy to see that $x_n^{-\epsilon}(\epsilon)-y_n^{-\epsilon}(\epsilon)\in \ker \pi$ for  $\epsilon=\pm1$ or $\pm$.
Let $H_1$ be the ideal of ${\mathcal U}_0(\mathfrak{g}_{2, tor})$ generated by $(x_n^{-\epsilon}(\epsilon)-y_n^{-\epsilon}(\epsilon))$ for  $\epsilon=\pm1$ or $\pm$.
Then $H_1 \subseteq \ker \pi$.

Denote the quotient algebra ${\mathcal U}_0(\mathfrak{g}_{2,tor})/H_1$ by ${\bar{\mathcal U}}_0(\mathfrak{g}_{2,tor})$.
Actually, in the quotient algebra ${\bar{\mathcal U}}_0(\mathfrak{g}_{2,tor})$, we have that

\begin{equation*}
\begin{split}
b_{n}(1)&=\gamma^{\frac{1}{2}}K_1^{-1}\,\bigl[\,x_{n}^+(0),\,y_n^{-}(1)\,\bigr]
=\gamma^{\frac{1}{2}}K_1^{-1}\,\bigl[\,x_{n}^+(0),\,x_n^{-}(1)\,\bigr]=a_{n}(1),\\
b_{n}(-1)&=\gamma^{\frac{1}{2}}K_1^{-1}\,\bigl[\,y_{n}^+(-1),\,y_n^{-}(0)\,\bigr]
=\gamma^{\frac{1}{2}}K_1^{-1}\,\bigl[\,x_{n}^+(-1),\,x_n^{-}(0)\,\bigr]=a_{n}(-1).
\end{split}
\end{equation*}

Then in the quotient algebra ${\bar{\mathcal U}}_0(\mathfrak{g}_{2, tor})$, by induction we have for $i=0, 1, \ldots, n$ that

\begin{equation*}
\begin{split}
y_i^{\pm}(\epsilon)&=\pm \gamma^{\pm\frac{1}{2}}\,\bigl[\,b_{i+1}(\epsilon),\,x_i^{\pm}(0)\,\bigr]=\pm \gamma^{\pm\frac{1}{2}}\,\bigl[\,a_{i+1}(\epsilon),\,x_i^{\pm}(0)\,\bigr]=x_i^{\pm}(\epsilon),\\
b_{i}(1)&=\gamma^{\frac{1}{2}}K_1^{-1}\,\bigl[\,x_{i}^+(0),\,y_i^{-}(1)\,\bigr]
=\gamma^{\frac{1}{2}}K_1^{-1}\,\bigl[\,x_{i}^+(0),\,x_i^{-}(1)\,\bigr]=a_{i}(1),\\
b_{i}(-1)&=\gamma^{\frac{1}{2}}K_1^{-1}\,\bigl[\,y_{i}^+(-1),\,x_i^{-}(0)\,\bigr]
=\gamma^{\frac{1}{2}}K_1^{-1}\,\bigl[\,x_{i}^+(-1),\,x_i^{-}(0)\,\bigr]=a_{i}(-1).
\end{split}
\end{equation*}
Arguing by the degree, one has that $y_i^{\pm}(k)=x_i^{\pm}(k)$ and $b_i(k)=a_i(k)$ for $k\in \mathbb{Z}^*$, which implies
$\ker \pi \subseteq H_1$.  Thus when $\mathfrak{g}$ is of type $A$, $${\mathcal U}_0(\mathfrak{g}_{2,tor})/H_1\cong U_{q}(\mathfrak {g}_{2,tor}).$$
When $\mathfrak{g}$ is not of type $A$,
$${\mathcal U}_0(\mathfrak{g}_{2,tor}) \cong U_{q}(\mathfrak {g}_{2,tor}).$$
 \end{proof}

\section{Proof of Theorem \ref{maintheo2-2}}
Theorem \ref{maintheo2-2} is equivalent to the following statement.
\begin{theo}\label{5.1}\, There exists an epimorphism $\pi_1: \mathcal{U}_0(\mathfrak{g}_{N,tor})\rightarrow \overline{U}_{q}(\mathfrak {g}_{N,tor})$ such that
$\pi_1(a)=a$  for $a\in \mathcal{U}_0(\mathfrak{g}_{N,tor})$.
\end{theo}
To show Theorem \ref{5.1}, we note that $\pi_1$ is an algebra homomorphism, as
$\pi_1$ preserve all relations from \eqref{n:comm0} to \eqref{n:comm13}
since they are part of the defining relations according to Definition \ref{defi-ntor}.

Now we are left to show $\pi_1$ is surjective. Actually we can define all generators of $\overline{U}_{q}(\mathfrak {g}_{N,tor})$ inductively in a similar manner as to Section 4. First of all, for $s\in J$, $k\in \mathbb{N}$ and  $\epsilon=\pm1$, we define (inductively): 
	\begin{gather*}
	a_0^{(s)}(1)=K_0^{-1}\gamma_s^{1/2}\,\bigl[\,x_0^+(\underline{0}),\,x_{0}^-(e_s)\,\bigr],\\
	a_0^{(s)}(-1)=K_0{\gamma_s}^{-1/2}\,\bigl[\,x_{{0}}^+(-e_s),\,x_{0}^-(\underline{0})\,
	\bigr],\\
x_0^{\pm}(\epsilon e_s)=\pm [2]_0^{-1}\gamma_s^{\pm\frac{1}{2}}\,\bigl[\,a^{(s)}_0(\epsilon),\,x_0^{\pm}(\underline{0})\,\bigr],\\
	x_{0}^{\pm}(\epsilon ke_s)=\pm [2]_0^{-1}\gamma_s^{\pm\frac{1}{2}}\,\bigl[\,a_0^{(s)}(\epsilon e_s),\,x_{0}^{\pm}(\epsilon (k-1)e_s)\,\bigr],\\
{\phi}_{0}^{(s)}(k)=(q_0-q_0^{-1})\gamma_s^{\frac{2-k}{2}}\,\bigl[\,x_{0}^+((k-1)e_s),\,x_{0}^{-}(e_s)\,\bigr],\\
{\varphi}_{0}^{(s)}(-k)=-(q_0-q_0^{-1})\gamma_s^{\frac{k-2}{2}}\,\bigl[\,x_{0}^+(-e_s),\,x_{0}^{-}((-k+1)e_s)\,\bigr].
\end{gather*}
Moreover, we can also define the elements $a_i^{(s)}(k)$ and $x_i^{\pm}(k e_s)$  for $i\in I$ and $k\in \mathbb{Z}^*$  by the same method in section 4.
Furthermore, we can construct that for $\epsilon, \epsilon' =\pm$ or $\pm1$ and $s\neq s'\in J$,
\begin{gather*}
x_0^{\pm}(\epsilon e_1+\epsilon'e_{2})=\pm[2]_0^{-1}\gamma_1^{\pm 1/2}\,\bigl[\,a_0^{(1)}(\epsilon),\,x_{0}^{\pm}(\epsilon' e_{2})\,\bigr].
\end{gather*}

Let $\underline{k}=k_1e_1+\cdots+k_{N-1}e_{N-1}$, we define that
\begin{gather*}
x_0^{\pm}(\underline{k})=\pm[2]_0^{-1}\gamma_{N-1}^{\pm 1/2}\,\bigl[\,a_0^{(N-1)}(k_{N-1}),\,x_{0}^{\pm}(k_1e_1+\cdots+k_{N-2}e_{N-2})\,\bigr].
\end{gather*}
Then we can define $x_i^{\pm}(\underline{k})$ using the same way. As a consequence, $\pi_1$ is surjective.

Let $H_2=\ker \pi_1$, then $\mathcal{U}_0(\mathfrak{g}_{N,tor})/H_2\cong U_{q}(\mathfrak {g}_{N,tor})$.
Therefore, we have completed the proof of Theorem \ref{maintheo2-2}.

\begin{remark}\, Unlike the case of $N=2$, it is complicated to describe $\ker \pi_1$. We have the following observations. \\
\noindent (1).  Notice that we can also construct that for $\epsilon, \epsilon' =\pm$ or $\pm1$, $i\in I$ and $s\neq s'\in J$,
\begin{gather*}
{x}_i^{\pm}(\epsilon e_s+\epsilon'e_{s'})=\pm[2]_i^{-1}\gamma_{s}^{\pm 1/2}\,\bigl[\,a_i^{(s)}(\epsilon),\,x_{i}^{\pm}(\epsilon' e_{s'})\,\bigr],\\
\dot{x}_i^{\pm}(\epsilon e_s+\epsilon'e_{s'})=\pm[2]_i^{-1}\gamma_{s'}^{\pm 1/2}\,\bigl[\,a_i^{(s')}(\epsilon'),\,x_{i}^{\pm}(\epsilon e_s)\,\bigr],
\end{gather*}
which represent the same elements in the algebra $U_{q}(\mathfrak {g}_{N,tor})$. So for $i\in I$ and $s\neq s'\in J$,
$$x_i^{-\epsilon}(\epsilon e_s+\epsilon'e_{s'})- \dot{x}_i^{-\epsilon}(\epsilon e_s+\epsilon'e_{s'})\in \ker \pi_1.$$	

\noindent (2). Note that in the algebra $U_{q}(\mathfrak {g}_{N,tor})$, we have that $[x_i^{\pm}(ke_s),x_{i}^{\pm}(le_{s'})\,]=0$, so
$$[x_i^{\pm}(ke_s),x_{i}^{\pm}(le_{s'})\,]\in \ker\pi_1.$$

\noindent (3). In type $A$, since the affine Dynkin diagram 
is a cycle, for $\epsilon=\pm1$ or $\pm$ and $s\in J$, we define that
\begin{gather*}
y_n^{\pm}(\epsilon e_s)=\pm \gamma_s^{\pm\frac{1}{2}}\,\bigl[\,a^{(s)}_0(\epsilon),\,x_n^{\pm}(\underline{0})\,\bigr],\\
b_n^{(s)}(1)=\gamma^{\frac{1}{2}}K_n^{-1}\,\bigl[\,x_{n}^+(\underline{0}),\,y_n^{-}(e_s)\,\bigr],\\
b_n^{(s)}(-1)=\gamma^{-\frac{1}{2}}K_n\,\bigl[\,y_{n}^+(-e_s),\,x_n^{-}(\underline{0})\,\bigr].
\end{gather*}
Moreover, we inductively define for $i=n-1, \cdots, 0$
\begin{gather*}
y_i^{\pm}(\epsilon e_s)=\pm \gamma_s^{\pm\frac{1}{2}}\,\bigl[\,b_{i+1}^{(s)}(\epsilon),\,x_i^{\pm}(\underline{0})\,\bigr],\\
b_{i}^{(s)}(1)=\gamma_s^{\frac{1}{2}}K_i^{-1}\,\bigl[\,x_{i}^+(\underline{0}),\,y_i^{-}(e_s)\,\bigr],\\
b_i^{(s)}(-1)=\gamma_s^{-\frac{1}{2}}K_i\,\bigl[\,y_{i}^+(-e_s),\,x_i^{-}(\underline{0})\,\bigr].
\end{gather*}
Furthermore, we also define that for $\underline{k}\in \mathbb{Z}^{N-1}$ and $i\in I$
\begin{gather*}
y_i^{\pm}(\underline{k})=\pm[2]_i^{-1}\gamma_{1}^{\pm 1/2}\,\bigl[\,b_i^{(N-1)}(k_{N-1}),\,x_{i}^{\pm}(k_1e_1+\cdots+k_{N-2}e_{N-2})\,\bigr].
\end{gather*}

For $\underline{k}\in \mathbb{Z}^{N-1}$, $\ell\in \mathbb{Z}^*$, $i\in I$ and $s\in J$
\begin{equation*}\label{ideal1}
x_i^{\pm}(\underline{k})-y_i^{\pm}(\underline{k})\in \ker\pi_1, \quad a_i^{(s)}(\ell)- b_i^{(s)}(\ell)\in \ker\pi_1.
\end{equation*}
\end{remark}

\section{Vertex representations of quantum $N$-toroidal algebras $ U_{q}(\mathfrak{g}_{N, tor})$  for simply-laced type}

In this section, we will construct a level-one vertex representation of the quantum $N$-toroidal algebra for simply-laced type via generating functions (c.f. Def. 3.2).

 Let
$I=\{0, 1, \cdots, n\}$ and $I_0=\{1, \cdots, n\}$. Let $\mathfrak{g}$ be a finite dimensional simple Lie algebra of simply-laced type over $\mathbb{K}$ with the Cartan matrix $(a_{ij})_{i,j\in I_0}$.  Denote by $\hat{\mathfrak{g}}$ the affine Kac-Moody Lie algebra associated to $\mathfrak{g}$
and its 
Cartan matrix by $(a_{ij})_{i,j\in I}$. Let $\mathfrak{h}$ and  $\hat{\mathfrak{h}}$ be their Cartan subalgebras,  
$\Delta$ and $\hat{\Delta}$ their root systems, respectively. Also let
$\Pi=\{\bar{\alpha}_1,\cdots, \bar{\alpha}_n\}$ be a basis of $\Delta$, where $\alpha_0,\,\alpha_1,\cdots, \alpha_n$ are the simple roots of $\hat{\mathfrak{g}}$.

Let $\bar{Q}=\bigoplus_{i=1}^{n}\mathbb{Z}\bar{\alpha}_i$
and $Q=\bigoplus_{i=0}^{n}\mathbb{Z}\alpha_i$ be the root lattices of $\mathfrak{g}$ and $\hat{\mathfrak{g}}$ respectively.
The affine weight lattice is
$P = \bigoplus_{i=0}^{n}\mathbb{Z}\Lambda_i\bigoplus \mathbb{Z}\delta$,
where $\Lambda_0, \cdots, \Lambda_n$ are the fundamental weights of $\hat{\mathfrak{g}}$ and $\delta$ is the baisc imaginary root.
	Let $U_{q}$($\widehat{\frak h}$) be the associative algebra generated by $\{\,a_i(l) \mid
	l\in\mathbb{Z}\backslash \{0\},\,i\in I\,\}$, satisfying the following relations for $m,\, l \in \mathbb{Z}\backslash \{0\}$,
	\begin{eqnarray}
	\label{6.15}
&& [\,a_i(m), a_j(l)\,]=\delta_{m+l, 0} \frac{[ma_{ij}]}{m}[m].
	\end{eqnarray}
The algebra $U_{q}$($\widehat{\frak h})$ is a Weyl algebra that deforms the enveloping algebra of the Heisenberg algebra.

We denote by
$U_{q}$($\widehat{\frak h}^+$) (resp.$U_{q}$($\widehat{\frak h}^-$) \,) the commutative subalgebra of
$U_{q}$($\widehat{\frak h}$) generated by $a_i(l)$ (resp.
$a_i(-l)$) with $l\in\mathbb{Z}_{>0}$, $i\in I, j\in J$. Let
S($\widehat{\frak h}^-$) be the symmetric algebra generated by $a_i(-l)$ with $l\in\mathbb{Z}_{>0}$. Then S($\widehat{\frak h}^-$) is a
$U_{q}$($\widehat{\frak h}$, $N$-tor)-module with the action defined by
\begin{gather*}
\gamma_s^{\pm\frac1{2}}\cdot v=q^{\pm\frac1{2}}\,v,\\
a_i^{(s)}(-l)\cdot v=a_i(-l)\,v,\\
a_i^{(s)}(l)\cdot v=\frac{[la_{ij}]}{l}\frac{q^{l}-q^{-l}}{q-q^{-1}} \frac{d\, v}{d\, a_i(-l)}.
\end{gather*}
for any $v\in \rm{S}$($\widehat{\frak h}^-$), $l\in\mathbb{Z}_{>0}$ and
$i\in I$.


Let $\varepsilon( \ \ , \ \ )$ be the 2-cocycle of the root lattice $Q$ such that
\begin{eqnarray*}
   \varepsilon(\alpha,\,\beta)&=&(-1)^{(\alpha,\,\beta)}\varepsilon(\beta,\,\alpha).
     \end{eqnarray*}

Let $\mathbb{K}\{Q\}=\sum\limits_{\alpha\in Q }\mathbb{K} e^{\alpha}$ be the twisted group algebra spanned by $e^{\alpha} \, (\alpha\in Q)$ with multiplication: $e^{\alpha}e^{\beta}=\varepsilon(\alpha,\beta)e^{\alpha+\beta}$.
Define the Fock space $\mathcal{F}=S(\widehat{\frak h}^-)\otimes \mathbb{K}\{Q\}$, and
extend the action of $U_{q}$($\widehat{\frak h}^+$) naturally.
The operators $K_i$, $q^{d}$ and $z^{a_i(0)}$ act on  $\mathcal{F}$ as follows ($v\otimes e^{\beta}\in \mathcal{F}$):
\begin{gather*}
e^{\alpha}(v\otimes e^{\beta})=v\otimes e^{\alpha}e^{\beta},\\
a_i(0)(v\otimes e^{\beta})={(\alpha_i,\,\beta)}v\otimes e^{\beta},\\
z^{a_i(0)}(v\otimes e^{\beta})=z^{(\alpha_i,\,\beta)}v\otimes e^{\beta},\\
q^{d}(v\otimes e^{\beta})=q^{m_0}(v\otimes e^{\beta}),
\end{gather*}
where $\beta=\sum\limits_{i=0}^{n}m_i\alpha_i\in Q$.
Let $ : ~~ :$ be the usual normal order defined by moving modes with lower degrees to the left and
$e^{\alpha_i}
a_i(0):=:a_i(0)e^{\alpha_i}:=e^{\alpha_i}a_i(0)$.

Now define the vertex operators:
$$\begin{array}{rcl}
&&Y_{i}^\pm(z)=\exp\left(\pm\sum\limits_{k=1}^\infty\dfrac{a_i(-k)}{[k]}q^{\mp k/2}z^{k}\right)\exp\left(\mp\sum\limits_{k=1}^\infty\dfrac{a_i(k)}{[k]}q^{\mp k/2}z^{-k}\right)\\
&& \hskip5cm\times e^{\pm \alpha_i}z^{\pm a_i(0)}=\sum\limits_{n\in \mathbb{Z}}Y_i^{\pm}(n)z^{-n},\\
&&\Phi_i(z) =  q^{a_i(0)}\exp \Big(
(q{-}q^{-1})\sum\limits_{\ell=1}^{\infty}
a_i(\ell)z^{-\ell}\Big), \\
&&\Psi_i(z)= q^{-a_i(0)}\exp
\Big({-}(q{-}q^{-1})
\sum\limits_{\ell=1}^{\infty}a_i(-\ell)z^{\ell}\Big).
\end{array}
$$


\begin{theo}\label{maintheo3}\, For $i\in I$ and $s\in J$, the  Fock
	space $\mathcal F$ is a level one $ U_q(\mathfrak{g}_{N, tor})$-module for simply-laced types under the action  $\rho$ defined by:
	$$\begin{array}{rcl}
		\gamma_s^{\pm\frac{1}{2}}&\mapsto& q^{\pm\frac{1}{2}},\\
	q^{\pm d} &\mapsto&  q^{\pm d}, \\
	K_i &\mapsto&  q^{a_i(0)},  \\
x_{i}^\pm(\underline{k})&\mapsto&Y_{i}^\pm(\rm{ht}(\underline{k})),\\
\phi_i^{(s)}(z) &\mapsto&\Phi_i(z) ,\\
	\varphi_i^{(s)}(z) &\mapsto& \Psi_i(z),
	\end{array}$$
where $\rm{ht}(\underline{k})\doteq k_1+\cdots+k_{N-1}$ for $\underline{k}=(k_1,\cdots, k_{N-1})$.
\end{theo}

This result can be checked directly by noting that the map specified in the theorem is in fact
a homomorphism from the quantum $N$-toroidal algebra to the quantum toroidal algebra
and then using the Fock space representation constructed in \cite{Sy}.

\section{Appendix}

In the appendix, we will list the Dynkin diagrams case by case according to the type of $\mathfrak g$
and GIM $M$ given in Definition \ref{defi3.10}. Here if $m_{ij}\in M$ such that $m_{ij}>0$ for $i\neq j$,
we use dotted lines to replace the edges of the Dynkin diagram for general Cartan matrix,
and we keep other rules of the Dynkin diagram for  Cartan matrix. We give the  Dynkin diagrams for the case of $N=2$ and $N=3$, respectively.

 \subsection{Dynkin diagrams for the case of $N=2$}

(I).\, Type $A_n (n>1)$:

 
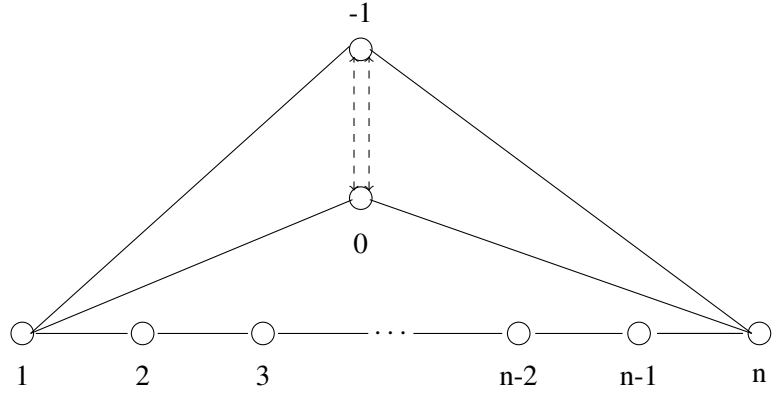
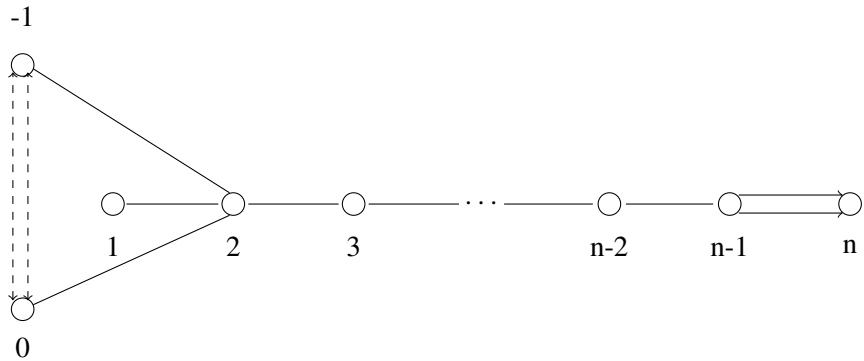
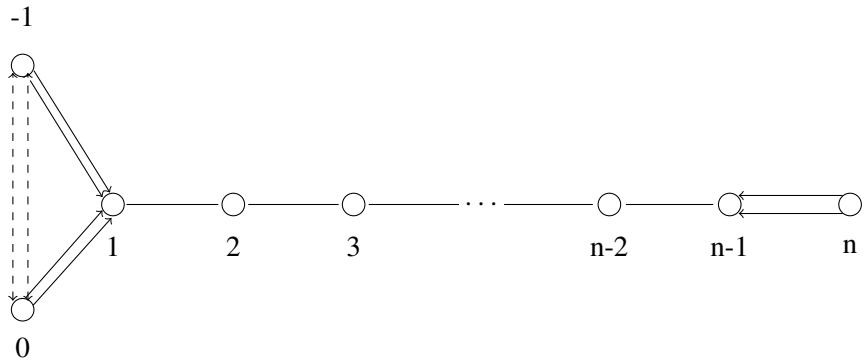
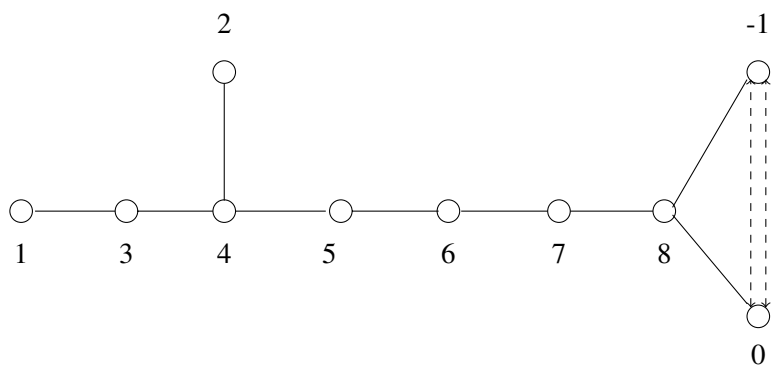
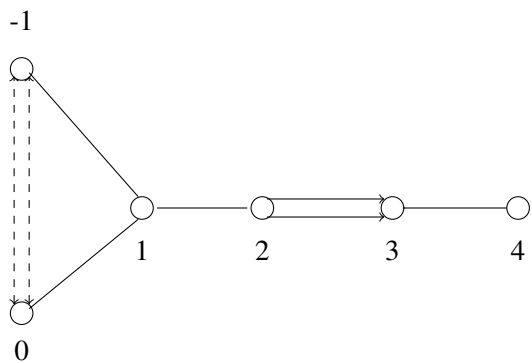
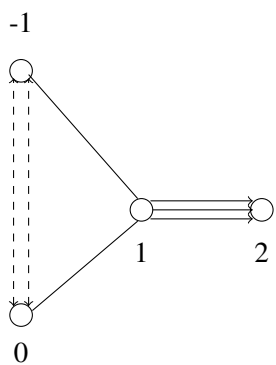
\begin{center}
 	\begin{tikzpicture}
 	\draw node at(5.1,1.2){0} (5.1,1.8)circle[radius=0.15];
 	\draw node at(5.1,4.28){-1} (5.1,3.78)circle[radius=0.15];
 	\draw[style=dashed] [<->=angle 90] (5.21,1.9)--(5.21,3.68);
 	\draw[style=dashed] [<->=angle 90] (5.01,1.9)--(5.01,3.68);
 	\draw (9.04,0)--(10.18,0);
 	\draw [=angle 30] (0.71,0)--(4.95,3.82);
 	\draw [=angle 18] (0.71,0)--(4.99,1.78);
 	\draw [=angle -18] (5.22,1.78)--(10.3,0);
 	\draw [=angle -30] (5.22,3.78)--(10.3,0);
 	\draw node{} node at(0.6,-0.56){1} (0.6,0)circle[radius=0.15] (0.78,0)--(2,0);
 	\draw node at(2.2,-0.56){2} (2.2,0)circle[radius=0.15] (2.4,0)--(3.6,0);
 	\draw  node at(3.8,-0.56){3} (3.8,0)circle[radius=0.15](4,0)--(5.2,0);
 	\draw  node at(5.53,0){$\cdots$} (5.8,0)--(6.98,0);
 	\draw node{} node at(7.2,-0.56){n-2} (7.2,0)circle[radius=0.15] (7.42,0)--(8.58,0);
 	\draw node at(8.8,-0.56){n-1} (8.8,0)circle[radius=0.15];
 	\draw (9.04,0)--(10.18,0);
 	\draw node at(10.4,-0.56){n} (10.4,0)circle[radius=0.15];
 	\end{tikzpicture}
 \end{center}

(II).\, Type $B_n (n>2)$:

 \begin{center}
 	\begin{tikzpicture}
 	\draw node at(-0.6,-1.9){0} (-0.6,-1.4)circle[radius=0.15];
 	\draw node at(-0.6,2.5){-1} (-0.6,1.85)circle[radius=0.15];
 	\draw[style=dashed] [<->=angle 90](-0.53,-1.28)--(-0.53,1.75);
 	\draw[style=dashed] [<->=angle 90](-0.73,-1.28)--(-0.73,1.75);
 	\draw  (-0.46,-1.34)--(2.15,-0.15);
 	\draw  (-0.46,1.8)--(2.15,0.15);
 	\draw node{} node at(0.6,-0.56){1} (0.6,0)circle[radius=0.15] (0.78,0)--(2,0);
 	\draw node at(2.2,-0.56){2} (2.2,0)circle[radius=0.15] (2.4,0)--(3.6,0);
 	\draw  node at(3.8,-0.56){3} (3.8,0)circle[radius=0.15](4,0)--(5.2,0);
 	\draw  node at(5.53,0){$\cdots$} (5.8,0)--(6.98,0);
 	\draw node{} node at(7.2,-0.56){n-2} (7.2,0)circle[radius=0.15] (7.42,0)--(8.58,0);
 	\draw node at(8.8,-0.56){n-1} (8.8,0)circle[radius=0.15];
 	\draw[->,>=left to] (8.92,0.12)--(10.3,0.12);
 	\draw[->,>=right to] (8.92,-0.12)--(10.3,-0.12);
 	\draw node at(10.4,-0.56){n} (10.4,0)circle[radius=0.15];
 	\end{tikzpicture}
 \end{center}

(III).\, Type $C_n ( n>1)$:

 \begin{center}
 	\begin{tikzpicture}
 	\draw node at(-0.6,-1.9){0} (-0.6,-1.4)circle[radius=0.15];
 	\draw node at(-0.6,2.5){-1} (-0.6,1.85)circle[radius=0.15];
 	\draw[style=dashed] [<->=angle 90](-0.53,-1.28)--(-0.53,1.75);
 	\draw[style=dashed] [<->=angle 90](-0.73,-1.28)--(-0.73,1.75);
 	\draw[->=angle](-0.56,-1.24)--(0.46,-0.08);
 	\draw[->=angle](-0.46,-1.34)--(0.58,-0.18);
 	\draw[->=angle](-0.50,1.65)--(0.46,0.1);
 	\draw[->=angle](-0.45,1.78)--(0.56,0.16);
 	\draw node{} node at(0.6,-0.56){1} (0.6,0)circle[radius=0.15] (0.78,0)--(2,0);
 	\draw node at(2.2,-0.56){2} (2.2,0)circle[radius=0.15] (2.4,0)--(3.6,0);
 	\draw  node at(3.8,-0.56){3} (3.8,0)circle[radius=0.15](4,0)--(5.2,0);
 	\draw  node at(5.53,0){$\cdots$} (5.8,0)--(6.98,0);
 	\draw node{} node at(7.2,-0.56){n-2} (7.2,0)circle[radius=0.15] (7.42,0)--(8.58,0);
 	\draw node at(8.8,-0.56){n-1} (8.8,0)circle[radius=0.15];
 	\draw[<-] (8.92,0.12)--(10.3,0.12);
 	\draw[<-] (8.92,-0.12)--(10.3,-0.12);
 	\draw node at(10.4,-0.56){n} (10.4,0)circle[radius=0.15];
 	\end{tikzpicture}
 \end{center}

(IV).\, Type $D_n (n>3)$:

 \begin{center}
 	\begin{tikzpicture}
 	\draw node at(-0.6,-1.9){0} (-0.6,-1.4)circle[radius=0.15];
 	\draw node at(-0.6,2.5){-1} (-0.6,1.85)circle[radius=0.15];
 	\draw[style=dashed] [<->=angle 90](-0.53,-1.28)--(-0.53,1.75);
 	\draw[style=dashed] [<->=angle 90](-0.73,-1.28)--(-0.73,1.75);
 	\draw  (-0.46,-1.34)--(2.15,-0.15);
 	\draw  (-0.46,1.8)--(2.15,0.15);
 	\draw node{} node at(0.6,-0.56){1} (0.6,0)circle[radius=0.15] (0.78,0)--(2,0);
 	\draw node at(2.2,-0.56){2} (2.2,0)circle[radius=0.15] (2.4,0)--(3.6,0);
 	\draw  node at(3.8,-0.56){3} (3.8,0)circle[radius=0.15](4,0)--(5.2,0);
 	\draw  node at(5.53,0){$\cdots$} (5.8,0)--(6.98,0);
 	\draw node{} node at(7.2,-0.56){n-2} (7.2,0)circle[radius=0.15] (7.42,0)--(8.58,0);
 	\draw node at(8.8,-0.56){n-2} (8.8,0)circle[radius=0.15];
 	\draw (8.92,0.12)--(10.3,1.75);
 	\draw (8.92,-0.12)--(10.3,-1.35);
 	\draw node at(10.4,2.5){n-1} (10.4,1.85)circle[radius=0.15];
 	\draw node at(10.4,-1.9){n} (10.4,-1.4)circle[radius=0.15];
 	\end{tikzpicture}
 \end{center}

(V).\, Type $E_6$:

 \begin{center}
 	\begin{tikzpicture}
 	\draw node at(3.8,4.8){0} (3.8,4.3)circle[radius=0.15];
 	\draw node at(7.2,4.8){-1} (7.2,4.3)circle[radius=0.15];
 	\draw[style=dashed] [<->=angle 90](3.9,4.4)--(7.1,4.4);
 	\draw[style=dashed] [<->=angle 90](3.9,4.2)--(7.1,4.2);
 	\draw  (5.48,1.95)--(3.73,4.2);
 	\draw  (5.58,1.95)--(7.25,4.2);
 	\draw node at(2.2,-0.56){1} (2.2,0)circle[radius=0.15] (2.4,0)--(3.6,0);
 	\draw  node at(3.8,-0.56){3} (3.8,0)circle[radius=0.15](4,0)--(5.35,0);
 	\draw  node at(5.53,-0.56){4} (5.53,0)circle[radius=0.15]
 	(5.8,0)--(6.98,0);
 	\draw node at(5.53,2.5){2} (5.53,1.85)circle[radius=0.15];
 	\draw(5.53,0.14)--(5.53,1.7);
 	\draw node{} node at(7.2,-0.56){5} (7.2,0)circle[radius=0.15] (7.42,0)--(8.58,0);
 	\draw node at(8.8,-0.56){6} (8.8,0)circle[radius=0.15];
 	\end{tikzpicture}
 \end{center}

(VI).\, Type $E_7$:

 \begin{center}
 	\begin{tikzpicture}
 	\draw node at(0.6,-1.9){0} (0.6,-1.4)circle[radius=0.15];
 	\draw node at(0.6,2.5){-1} (0.6,1.85)circle[radius=0.15];
 	\draw[style=dashed] [<->=angle 90](0.7,-1.28)--(0.7,1.75);
 	\draw[style=dashed] [<->=angle 90](0.5,-1.28)--(0.5,1.75);
 	\draw  (0.7,-1.34)--(2.15,-0.15);
 	\draw  (0.7,1.8)--(2.15,0.15);
 	\draw node at(2.2,-0.56){1} (2.2,0)circle[radius=0.15] (2.4,0)--(3.6,0);
 	\draw  node at(3.8,-0.56){3} (3.8,0)circle[radius=0.15](4,0)--(5.35,0);
 	\draw  node at(5.53,-0.56){4} (5.53,0)circle[radius=0.15]
 	(5.8,0)--(6.98,0);
 	\draw node at(5.53,2.5){2} (5.53,1.85)circle[radius=0.15];
 	\draw(5.53,0.14)--(5.53,1.7);
 	\draw node{} node at(7.2,-0.56){5} (7.2,0)circle[radius=0.15] (7.42,0)--(8.58,0);
 	\draw node at(8.8,-0.56){6} (8.8,0)circle[radius=0.15];
 	\draw(8.9,0)--(10.3,0);
 	\draw node at(10.4,-0.56){7} (10.4,0)circle[radius=0.15];
 	\end{tikzpicture}
 \end{center}

(VII).\, Type $E_8$:

 \begin{center}
 	\begin{tikzpicture}
 	\draw node at(-0.1,-0.56){1} (-0.1,0)circle[radius=0.15] (0.085,0)--(1.15,0);
 	\draw  node at(1.3,-0.56){3} (1.3,0)circle[radius=0.15](1.45,0)--(2.45,0);
 	\draw  node at(2.6,-0.56){4} (2.6,0)circle[radius=0.15]
 	(2.75,0)--(3.95,0);
 	\draw node at(2.6,2.5){2} (2.6,1.85)circle[radius=0.15];
 	\draw(2.6,0.14)--(2.6,1.7);
 	\draw node{} node at(4.0,-0.56){5} (4.15,0)circle[radius=0.15] (4.3,0)--(5.45,0);
 	\draw node at(5.58,-0.56){6} (5.58,0)circle[radius=0.15];
 	\draw(5.75,0)--(6.9,0);
 	\draw node at(7.05,-0.56){7} (7.05,0)circle[radius=0.15];
 	\draw node at(8.45,-0.56){8} (8.45,0)circle[radius=0.15];
 	\draw  (7.2,-0)--(8.3,0);
 	\draw node at(9.7,-1.9){0} (9.7,-1.4)circle[radius=0.15];
 	\draw node at(9.7,2.5){-1} (9.7,1.85)circle[radius=0.15];
 	\draw[style=dashed] [<->=angle 90](9.6,-1.28)--(9.6,1.75);
 	\draw[style=dashed] [<->=angle 90](9.8,-1.28)--(9.8,1.75);
 	\draw  (8.56,0.05)--(9.55,1.75);
 	\draw  (8.56,-0.05)--(9.55,-1.23);
 	\end{tikzpicture}
 \end{center}

(VIII).\, Type $F_4$:

 \begin{center}
 	\begin{tikzpicture}
 	\draw node at(0.6,-1.9){0} (0.6,-1.4)circle[radius=0.15];
 	\draw node at(0.6,2.5){-1} (0.6,1.85)circle[radius=0.15];
 	\draw[style=dashed] [<->=angle 90](0.7,-1.28)--(0.7,1.75);
 	\draw[style=dashed] [<->=angle 90](0.5,-1.28)--(0.5,1.75);
 	\draw  (0.7,-1.34)--(2.15,-0.15);
 	\draw  (0.7,1.8)--(2.15,0.15);
 	\draw node at(2.2,-0.56){1} (2.2,0)circle[radius=0.15] (2.4,0)--(3.6,0);
 	\draw  node at(3.8,-0.56){2} (3.8,0)circle[radius=0.15];
 	\draw[->] (3.86,0.12)--(5.42,0.12);
 	\draw[->] (3.86,-0.12)--(5.42,-0.12);
 	\draw  node at(5.53,-0.56){3} (5.53,0)circle[radius=0.15]
 	(5.68,0)--(7.05,0);
 	\draw node{} node at(7.2,-0.56){4} (7.2,0)circle[radius=0.15];
 	\end{tikzpicture}
 \end{center}

(IX).\, Type $G_2$:

 \begin{center}
 	\begin{tikzpicture}
 	\draw node at(0.6,-1.9){0} (0.6,-1.4)circle[radius=0.15];
 	\draw node at(0.6,2.5){-1} (0.6,1.85)circle[radius=0.15];
 	\draw[style=dashed] [<->=angle 90](0.7,-1.28)--(0.7,1.75);
 	\draw[style=dashed] [<->=angle 90](0.5,-1.28)--(0.5,1.75);
 	\draw  (0.75,-1.34)--(2.15,-0.15);
 	\draw  (0.7,1.8)--(2.15,0.15);
 	\draw node at(2.2,-0.56){1} (2.2,0)circle[radius=0.15];
 	\draw[->](2.32,0.12)--(3.68,0.12);
 	\draw[->](2.35,0)--(3.68,0);
 	\draw[->](2.32,-0.12)--(3.68,-0.12);
 	\draw  node at(3.8,-0.56){2} (3.8,0)circle[radius=0.15];
 	\end{tikzpicture}
 \end{center}

\subsection{Dynkin diagrams for the case of $N=3$}
(I).\, Type $A_n$, $(n>1)$:
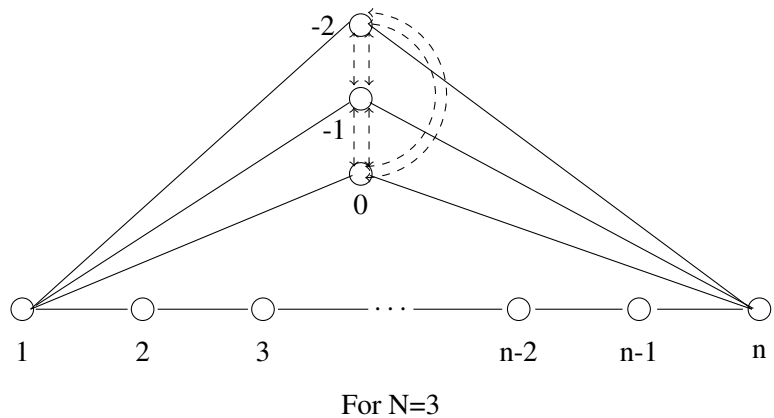
\begin{center}
	\begin{tikzpicture}
	\draw node at(5.1,1.4){0} (5.1,1.8)circle[radius=0.15];
	\draw node at(4.75,2.4){-1} (5.1,2.8)circle[radius=0.15];
	\draw node at(4.6,3.78){-2} (5.1,3.78)circle[radius=0.15];
	\draw[style=dashed] [<->=angle 90] (5.21,1.9)--(5.21,2.68);
	\draw[style=dashed] [<->=angle 90] (5.01,1.9)--(5.01,2.68);
	\draw[style=dashed] [<->=angle 90] (5.21,2.98)--(5.21,3.68);
	\draw[style=dashed] [<->=angle 90] (5.01,2.98)--(5.01,3.68);
	\draw (0.78,0.05)--(4.95,2.75);
	\draw (10.32,0.01)--(5.22,2.75);
	\draw [style=dashed][<->=angle 90](5.16,1.9)arc(-89:87:0.95);
	\draw [style=dashed][<->=angle 90](5.16,1.75)arc(-89:87:1.1);
	\draw (0.71,0)--(4.95,3.82);
	\draw (0.71,0)--(4.99,1.78);
	\draw (5.22,1.78)--(10.3,0);
	\draw (5.22,3.78)--(10.3,0);
	\draw node{} node at(0.6,-0.56){1} (0.6,0)circle[radius=0.15] (0.78,0)--(2,0);
	\draw node at(2.2,-0.56){2} (2.2,0)circle[radius=0.15] (2.4,0)--(3.6,0);
	\draw  node at(3.8,-0.56){3} (3.8,0)circle[radius=0.15](4,0)--(5.2,0);
	\draw  node at(5.53,0){$\cdots$} (5.8,0)--(6.98,0);
	\draw node{} node at(7.2,-0.56){n-2} (7.2,0)circle[radius=0.15] (7.42,0)--(8.58,0);
	\draw node at(8.8,-0.56){n-1} (8.8,0)circle[radius=0.15];
	\draw (9.04,0)--(10.18,0);
	\draw node at(10.4,-0.56){n} (10.4,0)circle[radius=0.15];
	\draw node{} node at(5.5,-1.26){For N=3};
	\end{tikzpicture}
\end{center}

(II). \, Type $B_n$, $(n>2)$:
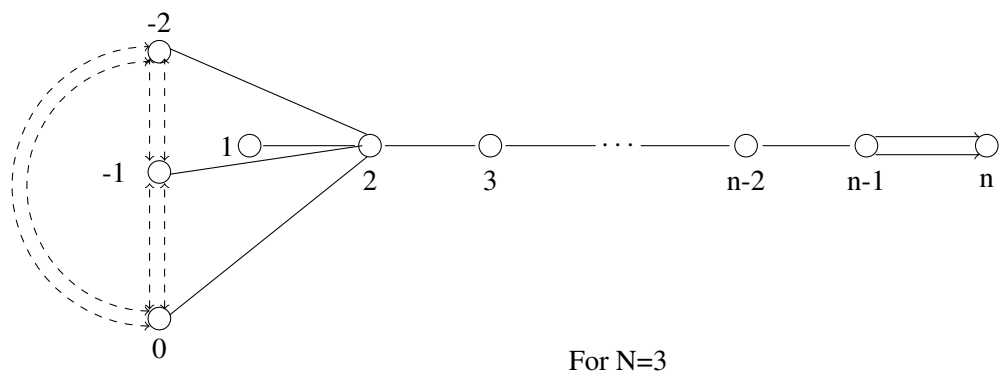
\begin{center}
	\begin{tikzpicture}
	\draw node at(-0.6,-2.7){0} (-0.6,-2.3)circle[radius=0.15];
	\draw node at(-0.6,1.65){-2} (-0.6,1.25)circle[radius=0.15];
	\draw node at(-1.2,-0.35){-1} (-0.6,-0.35)circle[radius=0.15];
	\draw[style=dashed] [<->=angle90](-0.53,-2.18)--(-0.53,-0.5);
	\draw[style=dashed] [<->=angle90](-0.73,-2.18)--(-0.73,-0.5);
	\draw[style=dashed] [<->=angle90](-0.53,-0.2)--(-0.53,1.16);
	\draw[style=dashed] [<->=angle90](-0.73,-0.2)--(-0.73,1.16);
	\draw [style=dashed][<->=angle 90](-0.73,-2.4)arc(269:91:1.86);
	\draw [style=dashed][<->=angle 90](-0.73,-2.2)arc(269:91:1.66);
	\draw  (-0.46,-0.38)--(2.1,0);
	\draw  (-0.46,-2.25)--(2.16,-0.15);
	\draw  (-0.46,1.29)--(2.16,0.15);
	\draw node{} node at(0.3,0){1} (0.6,0)circle[radius=0.15] (0.78,0)--(2,0);
	\draw node at(2.2,-0.46){2} (2.2,0)circle[radius=0.15] (2.4,0)--(3.6,0);
	\draw  node at(3.8,-0.46){3} (3.8,0)circle[radius=0.15](4,0)--(5.2,0);
	\draw  node at(5.53,0){$\cdots$} (5.8,0)--(6.98,0);
	\draw node{} node at(7.2,-0.46){n-2} (7.2,0)circle[radius=0.15] (7.42,0)--(8.58,0);
	\draw node at(8.8,-0.46){n-1} (8.8,0)circle[radius=0.15];
	\draw[->,>=left to] (8.92,0.12)--(10.3,0.12);
	\draw[->,>=right to] (8.92,-0.12)--(10.3,-0.12);
	\draw node at(10.4,-0.46){n} (10.4,0)circle[radius=0.15];
	\draw node{} node at(5.5,-2.86){For N=3};
	\end{tikzpicture}
\end{center}

(III).\, Type $C_n$, $(n>1)$:
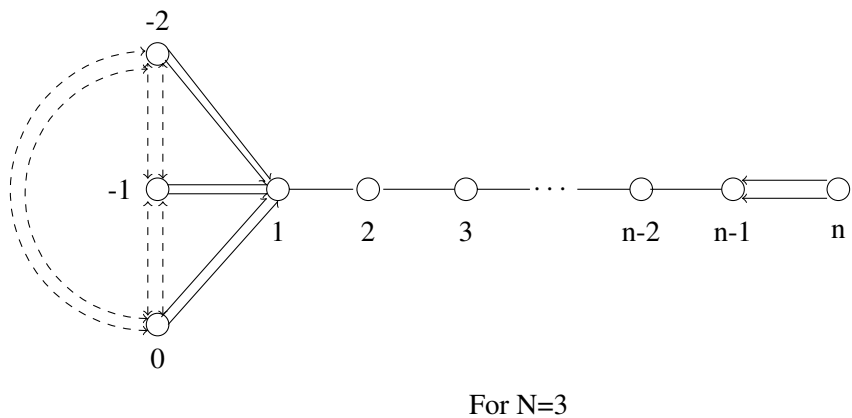
\begin{center}
	\begin{tikzpicture}
	\draw node at(-0.6,-2.25){0} (-0.6,-1.8)circle[radius=0.15];
	\draw node at(-0.6,2.25){-2} (-0.6,1.8)circle[radius=0.15];
	\draw[style=dashed] [<->=angle90](-0.53,-1.68)--(-0.53,-0.16);
	\draw[style=dashed] [<->=angle90](-0.73,-1.68)--(-0.73,-0.16);
	\draw node at(-1.1,0){-1} (-0.6,0)circle[radius=0.15];
	\draw[style=dashed] [<->=angle90](-0.53,0.14)--(-0.53,1.68);
	\draw[style=dashed] [<->=angle90](-0.73,0.14)--(-0.73,1.68);
	\draw [style=dashed][<->=angle 90](-0.73,-1.88)arc(269:92:1.86);
	\draw [style=dashed][<->=angle 90](-0.73,-1.72)arc(269:91:1.66);
	\draw[->,>=left to]  (-0.47,0.06)--(0.86,0.06);
	\draw [->,>=right to] (-0.46,-0.06)--(0.86,-0.06);
	\draw[->,>=left to]   (-0.55,-1.7)--(0.85,-0.12);
	\draw [->,>=right to]  (-0.46,-1.8)--(1.0,-0.16);
	\draw[->,>=left to]  (-0.46,1.86)--(0.90,0.13);
	\draw[->,>=right to]  (-0.50,1.72)--(0.83,0.10);
	\draw node{} node at(1.0,-0.56){1} (1.0,0)circle[radius=0.15] (1.15,0)--(2,0);
	\draw node at(2.2,-0.56){2} (2.2,0)circle[radius=0.15] (2.4,0)--(3.35,0);
	\draw  node at(3.5,-0.56){3} (3.5,0)circle[radius=0.15](3.65,0)--(4.35,0);
	\draw  node at(4.65,0){$\cdots$} (4.98,0)--(5.68,0);
	\draw node{} node at(5.83,-0.56){n-2} (5.83,0)circle[radius=0.15] (5.95,0)--(6.9,0);
	\draw node at(7.05,-0.56){n-1} (7.05,0)circle[radius=0.15];
	\draw[<-] (7.17,0.12)--(8.3,0.12);
	\draw[<-] (7.17,-0.12)--(8.3,-0.12);
	\draw node at(8.45,-0.56){n} (8.45,0)circle[radius=0.15];
	\draw node{} node at(4.2,-2.86){For N=3};
	\end{tikzpicture}
\end{center}

(IV).\, Type $D_n$, $(n>3)$:
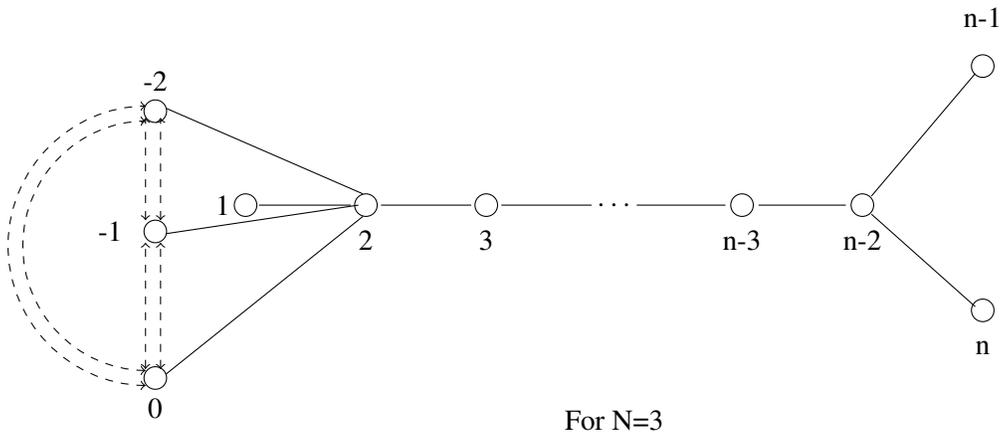
\begin{center}
	\begin{tikzpicture}
	\draw node at(-0.6,-2.7){0} (-0.6,-2.3)circle[radius=0.15];
	\draw node at(-0.6,1.65){-2} (-0.6,1.25)circle[radius=0.15];
	\draw node at(-1.2,-0.35){-1} (-0.6,-0.35)circle[radius=0.15];
	\draw[style=dashed] [<->=angle90](-0.53,-2.18)--(-0.53,-0.5);
	\draw[style=dashed] [<->=angle90](-0.73,-2.18)--(-0.73,-0.5);
	\draw[style=dashed] [<->=angle90](-0.53,-0.2)--(-0.53,1.16);
	\draw[style=dashed] [<->=angle90](-0.73,-0.2)--(-0.73,1.16);
	\draw [style=dashed][<->=angle 90](-0.73,-2.4)arc(269:91:1.86);
	\draw [style=dashed][<->=angle 90](-0.73,-2.2)arc(269:91:1.66);
	\draw  (-0.46,-0.38)--(2.1,0);
	\draw  (-0.46,-2.25)--(2.16,-0.15);
	\draw  (-0.46,1.29)--(2.16,0.15);
	\draw node{} node at(0.3,0){1} (0.6,0)circle[radius=0.15] (0.78,0)--(2,0);
	\draw node at(2.2,-0.46){2} (2.2,0)circle[radius=0.15] (2.4,0)--(3.6,0);
	\draw  node at(3.8,-0.46){3} (3.8,0)circle[radius=0.15](4,0)--(5.2,0);
	\draw  node at(5.53,0){$\cdots$} (5.8,0)--(6.98,0);
	\draw node{} node at(7.2,-0.46){n-3} (7.2,0)circle[radius=0.15] (7.42,0)--(8.58,0);
	\draw node at(8.8,-0.46){n-2} (8.8,0)circle[radius=0.15];
	\draw (8.92,0.12)--(10.3,1.75);
	\draw (8.92,-0.12)--(10.3,-1.35);
	\draw node at(10.4,2.5){n-1} (10.4,1.85)circle[radius=0.15];
	\draw node at(10.4,-1.9){n} (10.4,-1.4)circle[radius=0.15];
	\draw node{} node at(5.5,-2.86){For N=3};
	\end{tikzpicture}
\end{center}

(V).\, Type $E_6$ and $\tilde{J}=\{-N+1,\cdots, -1, 0, 1, \cdots, 6\}$:
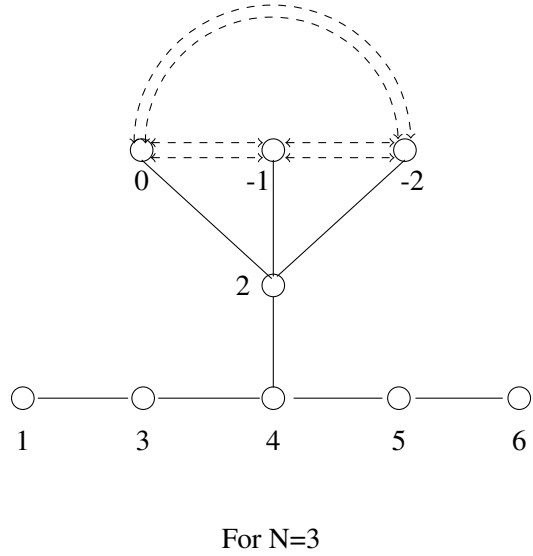
\begin{center}
	\begin{tikzpicture}
	\draw node at(3.78,2.9){0} (3.78,3.3)circle[radius=0.15];
	\draw node at(5.33,2.9){-1} (5.53,3.3)circle[radius=0.15];
	\draw node at(7.38,2.9){-2} (7.28,3.3)circle[radius=0.15];
	\draw[style=dashed] [<->=angle 90](3.9,3.4)--(5.39,3.4);
	\draw[style=dashed] [<->=angle 90](3.9,3.2)--(5.39,3.2);
	\draw[style=dashed] [<->=angle 90](5.69,3.2)--(7.14,3.2);
	\draw[style=dashed] [<->=angle 90](5.69,3.4)--(7.14,3.4);
	\draw [style=dashed][<->=angle 90](3.68,3.39)arc(180:2:1.84);
	\draw [style=dashed][<->=angle 90](3.83,3.39)arc(180:2:1.68);
	\draw  (5.51,1.59)--(3.78,3.17);
	\draw  (5.53,1.62)--(5.53,3.17);
	\draw  (5.58,1.62)--(7.28,3.17);
	\draw node at(2.2,-0.56){1} (2.2,0)circle[radius=0.15] (2.4,0)--(3.6,0);
	\draw  node at(3.8,-0.56){3} (3.8,0)circle[radius=0.15](4,0)--(5.35,0);
	\draw  node at(5.53,-0.56){4} (5.53,0)circle[radius=0.15]
	(5.8,0)--(6.98,0);
	\draw node at(5.13,1.5){2} (5.53,1.5)circle[radius=0.15];
	\draw(5.53,0.14)--(5.53,1.35);
	\draw node{} node at(7.2,-0.56){5} (7.2,0)circle[radius=0.15] (7.42,0)--(8.58,0);
	\draw node at(8.8,-0.56){6} (8.8,0)circle[radius=0.15];
	\draw node{} node at(5.5,-1.86){For N=3};
	\end{tikzpicture}
\end{center}

(VI).\, Type $E_7$ and $\tilde{J}=\{-N+1,\cdots, -1, 0, 1, \cdots, 7\}$:
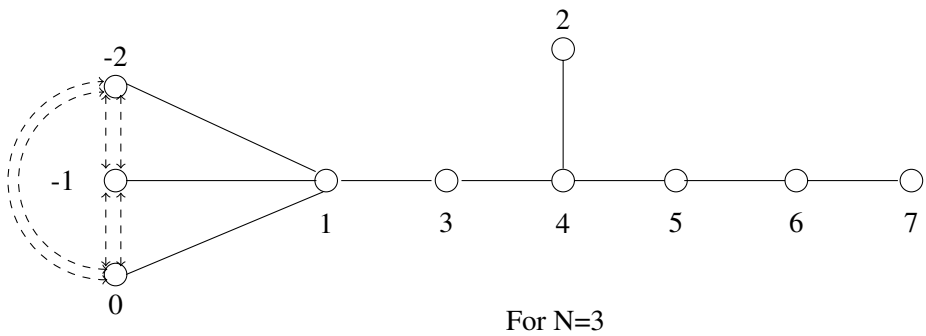
\begin{center}
	\begin{tikzpicture}
	\draw node at(-0.6,-1.65){0} (-0.6,-1.25)circle[radius=0.15];
	\draw node at(-0.6,1.65){-2} (-0.6,1.25)circle[radius=0.15];
	\draw node at(-1.3,0){-1} (-0.6,0)circle[radius=0.15];
	\draw[style=dashed] [<->=angle90](-0.53,-1.14)--(-0.53,-0.16);
	\draw[style=dashed] [<->=angle90](-0.73,-1.14)--(-0.73,-0.16);
	\draw[style=dashed] [<->=angle90](-0.53,0.16)--(-0.53,1.14);
	\draw[style=dashed] [<->=angle90](-0.73,0.16)--(-0.73,1.14);
	\draw [style=dashed][<->=angle 90](-0.71,-1.18)arc(270:92:1.18);
	\draw [style=dashed][<->=angle 90](-0.71,-1.32)arc(270:92:1.32);
	\draw  (-0.46,0)--(2.07,0);
	\draw  (-0.46,-1.25)--(2.16,-0.15);
	\draw  (-0.46,1.29)--(2.06,0.12);
	\draw node at(2.2,-0.56){1} (2.2,0)circle[radius=0.15] (2.4,0)--(3.6,0);
	\draw  node at(3.8,-0.56){3} (3.8,0)circle[radius=0.15](4,0)--(5.2,0);
	\draw  node at(5.35,-0.56){4} (5.35,0)circle[radius=0.15](5.5,0)--(6.7,0);
	\draw  node at(5.35,2.15){2} (5.35,1.75)circle[radius=0.15](5.35,0.16)--(5.35,1.6);
	\draw node at(6.85,-0.56){5} (6.85,0)circle[radius=0.15];
	\draw (6.96,0)--(8.3,0);
	\draw  node at(8.45,-0.56){6} (8.45,0)circle[radius=0.15](8.6,0)--(9.8,0);
	\draw node at(9.98,-0.56){7} (9.98,0)circle[radius=0.15];
	\draw node{} node at(5.25,-1.86){For N=3};
	\end{tikzpicture}
\end{center}

(VII).\, Type $E_8$ and $\tilde{J}=\{-N+1,\cdots, -1, 0, 1, \cdots, 8\}$:
\begin{center}
	\begin{tikzpicture}
	\draw node at(12.7,-1.65){0} (12.7,-1.25)circle[radius=0.15];
	\draw node at(12.7,1.65){-2} (12.7,1.25)circle[radius=0.15];
	\draw node at(13.2,0){-1} (12.7,0)circle[radius=0.15];
	\draw[style=dashed] [<->=angle90](12.77,-1.11)--(12.77,-0.16);
	\draw[style=dashed] [<->=angle90](12.57,-1.11)--(12.57,-0.16);
	\draw[style=dashed] [<->=angle90](12.77,0.16)--(12.77,1.14);
	\draw[style=dashed] [<->=angle90](12.57,0.16)--(12.57,1.14);
	\draw [style=dashed][<->=angle 90](12.85,-1.18)arc(-90:89:1.18);
	\draw [style=dashed][<->=angle 90](12.82,-1.32)arc(-90:89:1.32);
	\draw  (12.59,0)--(11.61,0);
	\draw  (12.59,-1.25)--(11.61,-0.12);
	\draw  (12.59,1.25)--(11.61,0.12);
	\draw node at(2.2,-0.56){1} (2.2,0)circle[radius=0.15] (2.4,0)--(3.65,0);
	\draw  node at(3.8,-0.56){3} (3.8,0)circle[radius=0.15](4,0)--(5.2,0);
	\draw  node at(5.35,-0.56){4} (5.35,0)circle[radius=0.15](5.5,0)--(6.7,0);
	\draw  node at(5.35,2.15){2} (5.35,1.75)circle[radius=0.15](5.35,0.16)--(5.35,1.6);
	\draw node at(6.85,-0.56){5} (6.85,0)circle[radius=0.15];
	\draw (6.96,0)--(8.3,0);
	\draw  node at(8.45,-0.56){6} (8.45,0)circle[radius=0.15](8.6,0)--(9.8,0);
	\draw node at(9.98,-0.56){7} (9.98,0)circle[radius=0.15](10.15,0)--(11.35,0);
	\draw node at(11.5,-0.56){8} (11.5,0)circle[radius=0.15];
	\draw node{} node at(7.8,-1.86){For N=3};
	\end{tikzpicture}
\end{center}

(VIII).\, Type $F_4$  and $\tilde{J}=\{-N+1,\cdots, -1, 0, 1, 3, 4\}$:
\begin{center}
	\begin{tikzpicture}
	\draw node at(-0.6,-1.65){0} (-0.6,-1.25)circle[radius=0.15];
	\draw node at(-0.6,1.65){-2} (-0.6,1.25)circle[radius=0.15];
	\draw node at(-1.3,0){-1} (-0.6,0)circle[radius=0.15];
	\draw[style=dashed] [<->=angle90](-0.53,-1.14)--(-0.53,-0.16);
	\draw[style=dashed] [<->=angle90](-0.73,-1.14)--(-0.73,-0.16);
	\draw[style=dashed] [<->=angle90](-0.53,0.16)--(-0.53,1.14);
	\draw[style=dashed] [<->=angle90](-0.73,0.16)--(-0.73,1.14);
	\draw [style=dashed][<->=angle 90](-0.71,-1.18)arc(270:92:1.18);
	\draw [style=dashed][<->=angle 90](-0.71,-1.32)arc(270:92:1.32);
	\draw  (-0.46,0)--(2.07,0);
	\draw  (-0.46,-1.25)--(2.16,-0.15);
	\draw  (-0.46,1.29)--(2.06,0.12);
	\draw node at(2.2,-0.56){1} (2.2,0)circle[radius=0.15] (2.4,0)--(3.65,0);
	\draw  node at(3.8,-0.56){2} (3.8,0)circle[radius=0.15];
	\draw[->,>=left to] (3.95,0.12)--(5.25,0.12);
	\draw[->,>=right to](3.95,-0.12)--(5.25,-0.12);
	\draw  node at(5.35,-0.56){3} (5.35,0)circle[radius=0.15](5.5,0)--(6.7,0);
	\draw node at(6.85,-0.56){4} (6.85,0)circle[radius=0.15];
	\draw node{} node at(3.2,-1.86){For N=3};
	\end{tikzpicture}
\end{center}

(IX).\, Type $G_2$ and $\tilde{J}=\{-N+1,\cdots, -1, 0, 1, 2\}$:
\begin{center}
	\begin{tikzpicture}
	\draw node at(-0.6,-1.65){0} (-0.6,-1.25)circle[radius=0.15];
	\draw node at(-0.6,1.65){-2} (-0.6,1.25)circle[radius=0.15];
	\draw node at(-1.3,0){-1} (-0.6,0)circle[radius=0.15];
	\draw[style=dashed] [<->=angle90](-0.53,-1.14)--(-0.53,-0.16);
	\draw[style=dashed] [<->=angle90](-0.73,-1.14)--(-0.73,-0.16);
	\draw[style=dashed] [<->=angle90](-0.53,0.16)--(-0.53,1.14);
	\draw[style=dashed] [<->=angle90](-0.73,0.16)--(-0.73,1.14);
	\draw [style=dashed][<->=angle 90](-0.71,-1.18)arc(270:92:1.18);
	\draw [style=dashed][<->=angle 90](-0.71,-1.32)arc(270:92:1.32);
	\draw  (-0.46,0)--(2.07,0);
	\draw  (-0.46,-1.25)--(2.16,-0.15);
	\draw  (-0.46,1.29)--(2.06,0.12);
	\draw node at(2.2,-0.56){1} (2.2,0)circle[radius=0.15];
	\draw  node at(3.8,-0.56){2} (3.8,0)circle[radius=0.15];
	\draw[->] (2.35,0.12)--(3.68,0.12);
	\draw[->] (2.35,0)--(3.68,0);
	\draw[->] (2.35,-0.12)--(3.68,-0.12);
	\draw node{} node at(2,-1.86){For N=3};
	\end{tikzpicture}
\end{center}

\medskip

\vskip30pt \centerline{\bf ACKNOWLEDGMENT}
Y. Gao would
like to thank the support of NSERC of Canada and  NSFC grant  11931009. N. Jing would like to thank the support of
Simons Foundation grant 523868 and NSFC grants 11531004 and 12171303.
L. Xia would
like to thank the support of NSFC grants 11871249 and 12171155.  H. Zhang would
like to thank the support of NSFC grant 11871325.
\bigskip

\bigskip

\noindent{\bf Statement on Conflict of interest}: The authors declare that they have no conflict of interest.

\bibliographystyle{amsalpha}

\begin{thebibliography}{ABC9}

\medskip

\bibitem[ABFP]{ABFP} B. Allison, S. Berman, J. Faulkner and A. Pianzola, \textit{ Multiloop realization of extended affine Lie algebras and Lie tori}, Trans. Amer. Math. Soc. \textbf{361} (2009), 4807-4842.

\bibitem[BM]{BM} S. Berman and R.V. Moody, \textit{ Lie algebras graded by finite root systems and the intersection matrix algebras of Slowdowy}, Invent. Math. \textbf{108} (1992), 323--347.



\bibitem[FJM1]{FJM1} B. Feigin, M. Jimbo, T. Miwa and E. Mukhin, \textit{Branching rules for quantum toroidal $\mathfrak{gl}(n)$},
Adv. Math. \textbf{300} (2016), 229-274. 

\bibitem[FJM2]{FJM2} B. Feigin, M. Jimbo, T. Miwa and E. Mukhin, \textit{Representations of quantum toroidal $\mathfrak{gl}_n$},
J. Algebra \textbf{380} (2013), 78--108. 

\bibitem[G]{G} Y.~Gao,
\textit{Involutive Lie algebras graded by finite root systems and compact forms of IM algebras},
Math. Zeitschrift \textbf{223} (1996), 651--672.

\bibitem[GHX]{GHX} Y.~Gao, N. ~Hu and L. Xia,
\textit{Quantized GIM algebras and their images in quantized Kac-Moody algebras}, Algebra Represent. Theory \textbf{24} (2021), 565--584.

\bibitem[GJ]{GJ} Y.~Gao and N. Jing,
\textit{$U_q(\mathfrak{gl}_N)$ action on $\mathfrak{gl}_N$-modules and quantum toroidal algebras}, J. Algebra \textbf{273} (2004), 
320--343.

\bibitem[GTL]{GTL} S.~Gautam and V.~Toledano-Laredo,
\textit{Yangians and quantum loop algebras},
Selecta Math. (N.S.) \textbf{19} (2013), 271--336.

\bibitem[GKV]{GKV} V.~Ginzburg, M.~Kapranov and E.~Vasserot, \textit{Langlands reciprocty for algebric surfaces,}  Math. Res. Lett. \textbf{2} (1995), 147--160.

\bibitem[GM]{GM} N. Guay and X. Ma, \textit{From quantum loop algebras to Yangians},
J. Lond. Math. Soc. (2) \textbf{86} (2012), 683--700.

\bibitem[GNW]{GNW} N. Guay, H. Nakajima and C. Wendlandt, \textit{Coproduct for the Yangian of an affine Kac-Moody algebra},
 Adv. Math. \textbf{338} (2018), 865--911. 



\bibitem[FJ]{FJ} I.~B.~Frenkel and N. Jing, \textit{Vertex representations of quantum affine algebras}, Proc. Nat'l. Acad. Sci. USA. \textbf{85} (1998), 9373--9377.

\bibitem[FJW]{FJW} I.~B.~Frenkel, N.~Jing and W.~Wang,
 \textit{Quantum vertex representations via finite groups and the McKay correspondence}, Comm. Math. Phys. \textbf{211} (2000), 365--393.

\bibitem[H1]{H1} D.~Hernandez, \textit{Drinfeld coproduct, quantum fusion tensor category and applications}, Proc. London Math. Soc. \textbf{95}
(2007), 567--608.

\bibitem[H2]{H2} D.~Hernandez, \textit{Quantum toroidal algebras and their representations},
Selecta Math. (N.S.) \textbf{14} (2009), 701--725.

\bibitem[J1]{J1} N.~Jing, \textit{On Drinfel'd realization of quantum
affine algebras},  Ohio State Univ. Math. Res. Inst. Publ. de
Gruyter, Berlin, \textbf{7}, 1998, pp.195--206.

\bibitem[J2]{J2} N. Jing, \textit{Quantum Kac-Moody algebras and vertex representations},
Lett. Math. Phys. \textbf{4} (1998), 261--271.

\bibitem[JZ1]{JZ1} N. Jing and H. Zhang,  \textit{Hopf algebraic structure of quantum toroidal algebra}, arXiv:1604.05416 .


\bibitem[JZ2]{JZ2} N. Jing and H. Zhang, \textit{Two-parameter twisted
quantum affine algebras}, J. Math. Phys. \textbf{57} (2016), 091702.

\bibitem[K]{K} S. Kolb, \textit{Quantum symmetric Kac-Moody pairs}, Adv. Math. \textbf{267} (2014), 395-469.

\bibitem[LT]{LT}  R. Lv and Y. ~Tan, \textit{On quantized generalized intersection matrix algebras associated to 2-fold affinization of Cartan matrices}, J. Algebra Appl. \textbf{12} (2013), 125--141.


\bibitem[M1]{M1} K. Miki, \textit{ Toroidal and level 0 $U_q'(\widehat{sl}_{n+1})$ actions on $U_q(\widehat{gl}_{n+1})$-modules} J. Math. Phys. \textbf{40} (1999), 3191--3210.

\bibitem[M2]{M2} K. Miki, \textit{Toroidal  braid  group  action  and  an  automorphism  of  toroidal algebra $U_q(sl_{n+1},tor)$}, Lett.Math. Phys. \textbf{47} (1999), 365--378.


\bibitem[M3]{M3} K. Miki, \textit{Representations of quantum toroidal algebra $U_q(sl_{n+1},tor)(n>2)$}, J. Math. Phys. \textbf{41} (2000), 7079--7098.

\bibitem[M4]{M4} K. Miki, \textit{Quantum toroidal algebra $U_q(sl_2,tor)$ and R matrices}, J. Math. Phys. \textbf{42} (2001), 2293--2308.

\bibitem[M5]{M5} K. Miki, \textit{Some quotient algebras arising from the quantum toroidal algebra
$U_q(sl_2(C_{\gamma}))$}. Osaka J. Math. \textbf{42} (2005), 885--929.

\bibitem[M6]{M6}  K. Miki,  \textit{Some quotient algebras arising from the quantum toroidal algebra
$U_q(sl_{n+1}(C_{\gamma}))$ $(n \geq 2)$}. Osaka J. Math. \textbf{43} (2006), 895--922.

\bibitem[Na1]{Na} H. Nakajima, \textit{Quiver varieties and quantum affine algebras}, translation of Sugaku \textbf{52} (2000), 337--359;
Sugaku Expositions \textbf{19} (2006), 53--78.

\bibitem[Na2]{Na2} H. Nakajima, \textit{Quiver varieties and finite-dimensional representations of quantum
affine algebras}, J. Amer. Math. Soc. \textbf{14} (2001), 145--238.


\bibitem[Ne]{Ne} E.~Neher,
\textit{Lie algebras graded by 3-graded root systems and Jordan pairs covered by grids},
Amer. J. Math. \textbf{118} (1996), 439--491.





\bibitem[RM]{RM} S. ~Rao and R. ~Moody, \textit{Vertex representations for $N$-toroidal Lie algebras and a generalization of the Virasoro algebra}, Comm. Math. Phys.,  \textbf{159} (1994), 239--264.

\bibitem[Sk]{Sk} K. ~Saito, \textit{Extended affine root systems. I. Coxeter transformations}, Publ. Res. Inst. Math. Sci. \textbf{21} (1985), 75--179.


\bibitem[Sy]{Sy} Y. ~Saito, \textit{Quantum toroidal algebras and their vertex representations}, Publ. RIMS. Kyoto Univ.  \textbf{34} (1998), 155--177.

\bibitem[STU]{STU}Y. ~Saito, K. ~Takemura and D. ~Uglov, \textit{Toroidal  actions  on  level1modules  of $U_q(\hat{sl}_n)$}, Transform. Groups \textbf{3} (1998), 75--102.

\bibitem[Sl]{Sl} P. ~Slodowy, 
\textit{Beyond Kac-Moody algebras, and inside}, Lie algebras and related topics (Windsor, Ont., 1984), pp. 361--371,
CMS Conf. Proc., \textbf{5}, Amer. Math. Soc., Providence, RI, 1986.

\bibitem[T1]{T1}  Y. ~Tan, \textit{Quantized GIM algebras and their Lusztig symmetries}, J. Algebra \textbf{289} (2005), 214--276.


\bibitem[T2]{T2}  Y. ~Tan, \textit{Drinfeld-Jimbo coproduct of quantized GIM
Lie algebras }, J. Algebra \textbf{313} (2007), 617--641.

\bibitem[VV1]{VV1}
 M. Varagnolo and E. Vasserot, \textit{Schur duality in the toroidal setting}, Comm. Math. Phys. \textbf{182} (1996), 469--484.

\bibitem[VV2]{VV2}
M. Varagnolo and E. Vasserot, \textit{Double-loop algebras and the Fock space}, Invent. Math. \textbf{133} (1998), 133--159.
	

\bibitem[X]{X}
L. Xia, \textit{Finite dimensional modules over quantum toroidal algebras}, Front. Math. China \textbf{15} (2020), 593--600.

\end{thebibliography}

\end{document}